\documentclass[10pt]{article}
\usepackage{amsmath,amssymb,amsfonts,amsthm}
\usepackage{mathrsfs}
\usepackage{indentfirst}
\usepackage[pdftex]{color,graphicx}
\usepackage[pdftex,bookmarks,unicode,colorlinks]{hyperref}
\usepackage{enumerate} 
\usepackage[numbers]{natbib} 
\usepackage{yfonts} 
\usepackage{stmaryrd} 
\usepackage{caption} 

\usepackage{tikz-cd}
\usetikzlibrary{matrix,arrows,decorations.pathmorphing}

\makeatletter
\let\@fnsymbol\@arabic
\makeatother

\allowdisplaybreaks[4]

\newtheorem{remark}{Remark}
\newtheorem{thm}{Theorem}[section]
\newtheorem{prop}[thm]{Proposition}

\newtheorem{cor}[thm]{Corollary}
\newtheorem{lem}[thm]{Lemma}
\newtheorem{defn}{Definition}[section]

\def\cC{{\mathcal C}}

\newcommand{\R}{\mathbb R}

\newcommand{\D}[1]{\mbox{\rm #1}}
\newcommand{\dd}{\D{d}}

\numberwithin{equation}{section}

\begin{document}

\title{\bf Singular Perturbations of Nonlocal HJB Equations in Multiscale Stochastic Control}

\author{\normalsize{\bf Qi Zhang$^{a,}\footnote{qzhang@bimsa.cn}$, 
Yanjie Zhang$^{b,}\footnote{zhangyj2022@zzu.edu.cn}$, 
and Ao Zhang$^{c,}\footnote{aozhang1993@csu.edu.cn}$} \\[10pt]
\footnotesize{${}^a$ Beijing Institute of Mathematical Sciences and Applications Yau Mathematical Sciences Center,
Tsinghua University, Beijing 100084, China.} \\[5pt]
\footnotesize{${}^b$ School of Mathematics and Statistics, Zhengzhou University 450001, China.} \\[5pt]
\footnotesize{${}^c$ School of Mathematics and Statistics, HNP-LAMA, Central South University, Changsha 410083,
China.} \\[5pt]
}

\date{}
\maketitle
\vspace{-0.3in}

\begin{abstract}
This paper investigates a class of multiscale stochastic control problems driven by $\alpha$-stable L\'evy noises, where the controlled dynamics evolve across separate slow and fast time scales. The associated value functions are governed by a family of nonlocal Hamilton-Jacobi-Bellman (HJB) equations subject to singular perturbations. By employing the perturbed test function method, we carefully analyze this singular perturbation problem and derive a limiting effective equation as the time-scale separation parameter $\varepsilon$ approaches zero. This limiting equation characterizes the value function of the averaged control problem, thereby establishing a rigorous averaging principle for the original multiscale system. The effective Hamiltonian-along with the corresponding averaged control problem¡ªis obtained by averaging with respect to the invariant measure of the fast process. Moreover, we provide a probabilistic proof of convergence and establish an explicit convergence rate for the value functions.
\end{abstract}


\section{Introduction.}

This paper investigates the averaging principle for controlled stochastic systems with slow-fast dynamics driven by non-Gaussian L\'evy noises. Specifically, we consider a class of jump-diffusion processes governed by the following slow-fast stochastic differential equations (SDEs) with $\alpha$-stable noises:
\begin{equation}
\label{slo0}
\left\{
\begin{aligned}
d{X^{\varepsilon}_s}&= b({X^{\varepsilon}_s},{Y^{\varepsilon}_s},v_{s})ds + d\textbf{L}_s^{{\alpha}_1}, \quad X^{\varepsilon}_{t} = x,\\
  d{Y^{\varepsilon}_s}& =\frac{1}{\varepsilon}c({X^{\varepsilon}_s}, Y^{\varepsilon}_s)ds + \varepsilon^{-\frac{1}{\alpha_2}}d\textbf{L}_s^{{\alpha_2}}, \quad Y^{\varepsilon}_{t}=y, \\
\end{aligned}
\right.
\end{equation}
where $0\leq t \leq s$, $X^\varepsilon_s \in \mathbb{R}^n$ and $Y^\varepsilon_s \in \mathbb{R}^m$ represent the slow and fast components, respectively, and $\{\textbf{L}^{{\alpha _i}}_{s}\}_{i=1,2}$ are independent symmetric $\alpha_i$-stable L\'evy processes with $1<\alpha_i <2$. The control variable $v_s$ takes values in a compact set $U$. Under suitable assumptions on the drift coefficients $b$ and $c$ (specified in Section \ref{sec-2}), the system is well-posed and the fast component $Y^\epsilon$ exhibits ergodic behavior. As the scaling parameter $\varepsilon \to 0$, the fast dynamics induce a singular perturbation on the slow component.

The associated stochastic control problem aims to maximize the discounted payoff functional:
\begin{equation}
\begin{aligned}
\label{mulSC}
J^{\varepsilon}(v,x,y,t)
:=\mathbb{E}\left[-\int^{T}_{t}e^{-\lambda(s-t)}L\left(X^{\varepsilon}_s,Y^{\varepsilon}_s, v_{s}\right)ds+e^{\lambda(t-T)}g\left(X^{\varepsilon}_{T},Y^{\varepsilon}_{T} \right) \middle| X^{\varepsilon}_t= x, Y^{\varepsilon}_{t}= y \right].
\end{aligned}
\end{equation}
The corresponding value function, defined as $u^{\varepsilon}(x,y,t) := \sup_{v \in \mathcal{U}} J^{\varepsilon}(v,x,y,t)$, satisfies a fully nonlinear nonlocal parabolic Hamilton-Jacobi-Bellman (HJB) equation on $(0,T) \times \mathbb{R}^n \times \mathbb{R}^m$. Here, $\lambda > 0$ denotes the discount rate, $L: \mathbb{R}^n \times \mathbb{R}^m \times U \to \mathbb{R}$ is the running cost, and $g: \mathbb{R}^n \to \mathbb{R}$ is the terminal cost.

Stochastic control problems arise in numerous applications, including financial mathematics \cite{AB88, HP09}, molecular dynamics \cite{CS12}, statistics \cite{PD12}, and materials science \cite{EA11}.
In the real world, many stochastic systems evolve on multiple time scales, often characterized by the coexistence of slow and fast dynamics.
The averaging principle provides a powerful tool to analyze such stochastic systems with singular perturbations by approximating the behavior of the slow component through an averaged dynamics.

For uncontrolled slow-fast stochastic systems, the averaging principle has been extensively studied.
In the Gaussian setting, the pioneering work of Khaminskii established this principle using discrete-time techniques. This foundation was further advanced by Pardoux and Veretennikov \cite{EP01}, who developed an averaging principle for diffusions via Poisson equation techniques. Subsequently, the averaging principle has been extensively explored in the context of stochastic partial differential equations by numerous researchers \cite{SC09, SC17}.

In the context of stochastic control, non-Gaussian jump noise and nonlocal HJB equations are widely used to model systems with discontinuous sample paths, such as intracellular transport \cite{LVS2015} and financial price dynamics \cite{M1963}.
For systems driven by L\'{e}vy noise, Bao, Yin, and Yuan \cite{JG18} derived averaging results for stochastic partial differential equations with $\alpha$-stable noise, while Sun, Xie, and Xie \cite{XS22} obtained both strong and weak convergence rates under various conditions.
More recently, the authors of \cite{YZ24} proved the weak convergence of the slow component to a L\'evy process as the scale-separation parameter tends to zero.

Given these developments for uncontrolled systems, it is both natural and important to investigate the asymptotic behavior of multiscale stochastic control problems.
For controlled stochastic systems, Borkar and Gaitsgory \cite{BG071, BG072} initiated the study of such problems using limit occupational measures and tightness arguments, and their approach has since been extended to infinite-dimensional control systems \cite{GC21, AS21}.
These works form a foundation for exploring the averaging principle for multiscale stochastic dynamics with singular perturbations.

By the dynamic programming principle (see, e.g., \cite{LP09, HP98}), a multiscale stochastic control problem can be equivalently characterized through its associated Hamilton-Jacobi-Bellman (HJB) equation. More precise, the value function $u^{\varepsilon}$ satisfies the following nonlocal HJB equation in the viscosity sense:
\begin{equation}
\label{HJBu}
  \left\{
   \begin{aligned}
   & \partial_t u^{\varepsilon}= \left(-\Delta_{x}\right)^{\frac{\alpha_1}{2}}u^{\varepsilon}-\frac{1}{\varepsilon}\mathcal{L}_{2}u^{\varepsilon} -H(x,y,\nabla_{x}u^{\varepsilon}) + \lambda u^{\varepsilon}, \\
   & u^{\varepsilon}(T,x,y) =g(x,y),
   \end{aligned}
   \right.
\end{equation}
where the Hamiltonian defined by
\begin{equation}
    H(x,y,p_{x}) :=\sup_{v\in {U}}\left[b(x,y,v)\cdot p_{x} -L(x,y,v)\right],
\end{equation}
and
\begin{align*}
    -\left(-\Delta\right)^{\frac{\alpha_{1}}{2}}u(\cdot) = & \int_{\mathbb{R}^{n}} \left[u(\cdot+z)-u(\cdot)-\mathbf{1}_{B_{1}(0)}(z)\langle \nabla u(\cdot), z\rangle \right]\frac{C_{n,\alpha_{1}}}{|z|^{n+\alpha_{1}}}dz\quad \nonumber
\end{align*}
is the infinitesimal generator of $(X_s^\varepsilon, Y_s^\varepsilon)$. We recall that the fractional Laplacian is defined as
\begin{align*}
    -\left(-\Delta\right)^{\frac{\alpha_{1}}{2}}u(\cdot) = & \int_{\mathbb{R}^{n}} \left[u(\cdot+z)-u(\cdot)-\mathbf{1}_{B_{1}(0)}(z)\langle \nabla u(\cdot), z\rangle \right]\frac{C_{n,\alpha_{1}}}{|z|^{n+\alpha_{1}}}dz\quad \nonumber
\end{align*}
with the normalization constant $C_{n,\alpha_{1}}=\frac{\alpha\Gamma(n+\alpha_{1})/2}{2^{1-\alpha_{1}}\pi^{n/2}\Gamma(1-\alpha_{1}/2)}$.

From an analytical viewpoint, the small-scale parameter in a multiscale stochastic control problem \eqref{mulSC} naturally induces a singular perturbation structure in the associated nonlocal HJB equations \eqref{HJBu}.
Consequently, the averaging principle for multiscale stochastic control can be equivalently formulated as a singular perturbation problem for the corresponding family of HJB equations.
The study of such singular perturbations has a rich history within the viscosity solution framework for HJB equations.
Seminal works by Alvarez and Bardi \cite{AB03}, followed by subsequent developments by Barles, Hama, and their collaborators \cite{BH23, BH24, BCM10}, established convergence results for singularly perturbed local and nonlocal HJB equations by employing the perturbed test function method, originally introduced by Evans \cite{E89} in the context of periodic homogenization.
Building on this line of research, \cite{BM16} extended the analysis to settings involving non-Gaussian fast processes, thereby covering a broader class of stochastic dynamics.
More recently, Gassiat and Manca \cite{GM23} quantified the rate of convergence of the perturbed value functions toward their effective limits in unbounded domains.

Given these developments, it is both natural and important to investigate the asymptotic behavior of multiscale stochastic control problems and their associated singularly perturbed HJB equations.
However, averaging principles for controlled slow-fast SDEs driven by $\alpha$-stable L'evy noise and the corresponding singular perturbations of nonlocal HJB equations remain relatively unexplored.
Extending the averaging principle to multiscale jump diffusions presents two major challenges.
From a probabilistic viewpoint, sample paths under non-Gaussian L\'evy noise are c\'adl\'ag (right-continuous with left limits), which complicates the application of classical time-discretization and martingale methods typically used in Gaussian settings.
Analytically, the governing HJB equations involve nonlocal operators, and the definition of their viscosity solutions is itself nonlocal in nature.
These two features trajectory discontinuity and operator nonlocality create significant technical difficulties that require new analytical tools and probabilistic arguments.

In this work, we address these challenges by developing an averaging principle for the multiscale control problem \eqref{mulSC} and the associated singular perturbation problem for the nonlocal HJB equation \eqref{HJBu}. We first analysis this problem from the PDE standpoint.
By employing the perturbed test function method with the Liouville property, we rigorously prove the convergence of $u^{\varepsilon}$ to an effective value function, thereby extending the framework of \cite{BH23} to the nonlocal setting.
This result characterizes the asymptotic behavior of the value function $u^{\varepsilon}$ as the scale-separation parameter $\varepsilon \to 0$.
Then from a probabilistic standpoint, we further establish a rigorous asymptotic analysis of the multiscale stochastic control problem \eqref{mulSC} and identify the effective averaging equation of \eqref{slo0} as $\varepsilon \to 0$.
This limiting equation characterizes the averaging principle for the original multiscale stochastic control system.
Moreover, the effective stochastic control problem associated with the effective HJB equation are constructed by averaging with respect to the invariant measure of the fast component $Y^{\varepsilon}_{s}$.
Our proofs rely on nonlocal Poisson equation techniques to establish the convergence of the value functions.
Beyond proving convergence, we also derive an explicit convergence rate, which significantly strengthens the quantitative understanding of the averaging behavior in multiscale stochastic control problems.

This work is organized as follows. We begin in Section 2 by stating the main theorems and introducing the underlying assumptions, which encompass regularity and dissipativity conditions on the drift coefficients $b$ and $c$, and hypotheses on the terminal function $g$ and running cost $L$. The convergence of viscosity solutions for the singularly perturbed HJB equations is then analyzed in Section 3 using the perturbed test function method. Section 4 shifts to a probabilistic perspective, where the strong averaging principle for system \eqref{slo0} is established, and the convergence of the value functions is proven via techniques from the Poisson equation. We conclude with general remarks in Section 5. Supplementary material is included in the appendices: a proof of Lemma \ref{LiouP} is provided in Appendix A, while a heuristic derivation based on the dynamic programming principle is presented in Appendix B.

We conclude this section with a summary of frequently used notations. The letter $C$ denotes a generic positive constant whose value may change from line to line. We write $C(p)$ indicate dependence on a parameter $p$. Define the parameter set
$$\mathcal{P} = \left\{ \|c\|_2, \beta, \|\nabla_x b(x,y,v)\|_0, \|\nabla_y\nabla_x b(x,y,v)\|_0, \|\nabla^2_y b(x,y,v)\|_0, \|\nabla_y \nabla_v b\|_{0}, T \right\}.$$
and denote the constant by $ C(\mathcal{P})$ for clarity.
We use $\otimes$, $\langle \cdot, \cdot \rangle $, and $|\cdot|$ to denote the tensor product, inner product, and norm in Euclidean space, respectively. The gradient operator in Euclidean space is denoted by $\nabla$. For positive integers $k,l$ and a probability measure $\mu$, we define the function spaces:
\begin{equation*}
\begin{aligned}
&\mathcal{B}_{b}(\mathbb{R}^{n}):=\left\{f: \mathbb{R}^{n}\rightarrow \mathbb{R}\mid f\text{ is bounded Borel measurable}\right\},\\
&\mathcal{C}_{0}(\mathbb{R}^{n}):=\left\{f: \mathbb{R}^{n}\rightarrow \mathbb{R}\mid f\text{ is continuous and has compact support}\right\},\\
&\mathcal{C}^{\mu}_{0}(\mathbb{R}^{n}):=\left\{f \in\cC_{0}(\R^n)\mid f\text{ is centered with respect to the measure $\mu$, i.e., $\textstyle{\int_{\mathbb{R}^{n}} f(x) \mu(dx)} = 0$}\right\},\\
&\cC^{k}(\mathbb{R}^{n}):=\left\{f: \mathbb{R}^{n}\rightarrow \mathbb{R}\mid f\text{ and all its partial  derivatives up to order $k$ are continuous} \right\},\\
&\cC^{k}_{b}(\mathbb{R}^{n}):=\left\{f \in \cC^{k}(\mathbb{R}^{n})\mid\text{ for  $ 1\leq i \leq k$, the $i$-th order partial  derivatives of $f$ are bounded} \right\},\\
&\cC^{k,l}_{b}(\mathbb{R}^{n}\times \mathbb{R}^{m} ):=\left\{f(x,y)\mid \text{ for $1\leq |\beta_1| \leq k$ and $1\leq |\beta_2|\leq l$, $\nabla^{\beta_1}_{x}\nabla^{\beta_2}_{y}f $ is uniformly bounded} \right\}.
\end{aligned}
\end{equation*}
We endow the space $\cC^{k}(\mathbb{R}^{n})$ with the norm $\|h\|_{k}=\|h\|_{0}+\sum_{j=1}^k\|\nabla^{\otimes j}h\|$, making it a Banach space. Finally, we define the conditional expectation:
\begin{equation*}
    \mathbb{E}_{(x,y)}[\cdot] := \mathbb{E}\left[\; \cdot \; \big| \; X^{\varepsilon, t,x}_{t}=x,\;Y^{\varepsilon, t,y}_{t}=y\right].
\end{equation*}

\section{Assumptions and the main results}\label{sec-2}

To ensure the well-posedness of the slow-fast stochastic differential equation (SDE) (\ref{slo0}) and the ergodicity of the fast process $Y^{\varepsilon}_{t}$, we impose the following regularity and dissipativity conditions on the drift coefficients $b$ and $c$.

\noindent ($\bf{A_{b}}$): Assume that the function $b$ is Lipschitz continuous in $(x, y)$ uniformly with respect to
$v\in {U}$, and has linear growth in $x$. That is, there exist positive constants  $K_1, K_2, K_3$ such that for all $x,x_1, x_2 \in \mathbb{R}^{n}$, $y,y_1, y_2 \in \mathbb{R}^{m}$, and $v\in {U}$,
\begin{equation*}
| b(x_1,y_1, v)-b(x_2,y_2,v)| \leq K_1 ( | x_1-x_2| +| y_1-y_2| ),
\end{equation*}
and
\begin{equation*}
\left|b(x,y,v)\right|\leq K_2 (1+|x|), ~~\sup_{v\in U}\left|\nabla_yb(x,y,v)\right|\leq K_3.
\end{equation*}

\noindent ($\bf{A_{c}}$): Suppose that $ c\in \mathcal{C}^{2,3}_{b}(\mathbb{R}^{n}\times \mathbb{R}^{m} )$, and there exists a positive constant $\beta$ such that for all $y_1, y_2 \in \mathbb{R}^{n}$,
\begin{equation}
\label{0disb}
  \sup_{x\in \mathbb{R}^{n}} |c(x,0)|< \infty, \quad \sup_{x\in \mathbb{R}^{n}} \langle c(x,y_1)-c(x,y_2), y_1-y_2 \rangle \le -\beta |y_1-y_2|^2.
\end{equation}

\begin{remark}
Under assumptions ($\bf{A_{b}}$) and ($\bf{A_{c}}$), the stochastic differential equation (\ref{slo0}) has a unique strong solution $({X^{\varepsilon}},{Y^{\varepsilon}})$, see e.g. \cite{DA09}.
\end{remark}

To study the stochastic control problem associated with the slow-fast jump-diffusion system (\ref{slo0}), we also impose the following conditions on the utility function $g$ and running cost $L$.\\
\noindent ($\bf{A_{L}}$): The running cost $L$ is uniformly Lipschitz continuous in $x$, with uniform continuity in $v$.  Moreover, there exists a constant $K_4>0$ so that
\begin{equation}
    \sup_{y \in \mathbb{R}^m, v\in {U}}|L(x,y,v)| \leq K_4(1+|x|),\quad \sup_{x \in \mathbb{R}^n, v\in U}|\nabla_{y} L(x,y,v)| \leq K_4.
\end{equation}
\noindent ($\bf{A_{g}}$):
The utility function $g$ is Lipschitz continuous in $x$ and $y$. Moreover, there exists a constant $K_5>0$ so that
\begin{equation}
    \sup_{y\in \mathbb{R}^m}|g(x,y)| \leq K_5(1+|x|).
\end{equation}

\begin{remark}
Under assumptions $(\bf{A_{b}})$ and $(\bf{A_{L}})$, the Hamiltonian
\begin{equation}
    H(x,y,p_{x}) :=\sup_{v\in {U}}\left[b(x,y,v)\cdot p_{x} -L(x,y,v)\right]
\end{equation}
is Lipschitz continuous in $x$, $y$, and $p_{x}$, and convex in $p_{x}$.
\end{remark}

Now we introduce the following frozen equation:

\begin{equation}\label{0nsde}
dY^{x^{\prime},y}_{s} = c(x^{\prime},Y^{x^{\prime},y}_{s}) ds + d\textbf{L}^{\alpha_2}_{s}, \quad Y^{x^{\prime},y}_{0} = y \in \mathbb{R}^m.
\end{equation}
Under the dissipative assumption ($\mathbf{A_{c}}$), the jump process $Y^{x^{\prime},y}_{s}$ admits a unique ergodic measure $\mu^{x^{\prime}}$, as shown in, e.g., \cite[Lemma 2]{YZ24}.
Based on this ergodic measure $\mu^{x^{\prime}}$, we define the effective Hamiltonian
\begin{equation}
\label{OH}
\Bar{H}(x^{\prime},p_{x^{\prime}}) = \int_{\mathbb{R}^m} H(x^{\prime},y,p_{x^{\prime}}) , \mu^{x^{\prime}}(dy),
\end{equation}
and the effective HJB equation:
\begin{equation}
\label{0EHJBu}
\left\{
\begin{aligned}
& \partial_t u(t,x) = (-\Delta_{x})^{\frac{\alpha_1}{2}} u(t,x) - \Bar{H}(x,\nabla_{x} u(t,x)) + \lambda u(t,x), \\
& u(T,x) = \Bar{g}(x).
\end{aligned}
\right.
\end{equation}

In the next result, we show that the value function $u^{\varepsilon}$ converges to the viscosity solution of the effective HJB equation.

\begin{thm}\label{theorem1}
Under assumptions ($\mathbf{A_{b}}$), ($\mathbf{A_{c}}$), ($\mathbf{A_{g}}$), and ($\mathbf{A_{L}}$), the value function $u^{\varepsilon}$ of the multiscale stochastic control problem (\ref{mulSC}) converges uniformly to the viscosity solution $\Bar{u}$ of the effective HJB equation \eqref{0EHJBu}. That is, for every $x \in \mathbb{R}^n$, $y \in \mathbb{R}^m$, we have
\begin{equation*}
\lim_{\varepsilon \to 0} u^{\varepsilon}(t,x,y) = u(t, x).
\end{equation*}
\end{thm}

Now we consider the associated effective limiting stochastic control problem for the effective HJB equation \eqref{0EHJBu}, which can be interpreted as an averaging of the original multiscale stochastic control problem \eqref{mulSC}.

Since the integral and supremum operators cannot be interchanged, the effective Hamiltonian defined in \eqref{OH} is not of Bellman type. Hence, $\Bar{H}$ cannot be expressed in the form derived from Dynamic Programming, i.e.,
\begin{equation*}
\label{HH}
\Bar{H}(x^{\prime},p_{x^{\prime}}) := \int_{\mathbb{R}^m} H(x^{\prime},y,p_{x^{\prime}}) , \mu^{x^{\prime}}(dy) = \int_{\mathbb{R}^m} \sup_{v \in U} \left[ b(x^{\prime},y,v) \cdot p_{x^{\prime}} - L(x^{\prime},y,v) \right] \mu^{x^{\prime}}(dy).
\end{equation*}
To ensure that the effective Hamiltonian defined in \eqref{HH} is of Bellman type, the original control set $U$ must be extended to ${U}^{ex}:=\{\vartheta: \mathbb{R}^{m}\rightarrow U \text{measurable}\}$. Under this extended control set, the effective Hamiltonian \eqref{OH} is of Bellman type and satisfies
\begin{equation}
\label{HHH}
\Bar{H}(x^{\prime},p_{x^{\prime}}) = \sup_{\vartheta \in \mathcal{U}} \int_{\mathbb{R}^m} \left[ b(x^{\prime},y,\vartheta) \cdot p_{x^{\prime}} - L(x^{\prime},y,\vartheta) \right] \mu^{x^{\prime}}(dy).
\end{equation}
For a detailed illustration of \eqref{HHH}, we refer to \cite[Proposition 3.1]{BH24}.

To obtain the convergence rate of the original system¡¯s value functions, stronger assumptions on the coefficients $b, c, L, g$ are required.\\

\noindent ($\bf{A^{'}_{b}}$): Assume that the function $b\in \cC^{2,3}_{b}$.\\

\noindent ($\bf{A^{'}_{L}}$): Suppose that the function $L\in \cC^{2,3}_{b}$.\\

\noindent ($\bf{A^{'}_{g}}$):
The utility function $g \in \cC^{3}_{b}$ does not depend on the variable $y$; that is, $g(x,y)=g(x)$.\\

\noindent ($\bf{A_{b,L}}$): The function $b$ and the running cost $L$ are decoupled from the fast variables, i.e.,
$b(x,y,v)=b_1(x,v)+b_2(x,y), ~~L(x,y,v)=L_1(x,v)+L_2(x,y)$.\\

We can now define the effective stochastic control problem
\begin{equation}
\label{EffSC}
    u(t,x)=\sup_{\vartheta \in \mathcal{U}}\Bar{J}(t,x,v):=\mathbb{E}\left( - \int^{T}_{t} e^{-\lambda(s-t)}\Bar{L}(\Bar{X}_{s}, v_s) ds + e^{\lambda(t-T)}{g}(\Bar{X}_{T}) | \Bar{X}_{t} = x \right),
\end{equation}
where $\Bar{X}$ is the solution to the effective system \eqref{barX}.

The following theorem is the main result of this paper.
\begin{thm}\label{intm}
Let the assumptions $(\bf{A^{'}_{b}})$, $(\bf{A_{c}})$, $(\bf{A^{'}_{L}})$, $(\bf{A^{'}_{g}})$ and ($\bf{A_{b,L}}$) hold.
Let $v$ be any admissible control, let $(X^{\varepsilon},Y^{\varepsilon})$ be the solution
of \eqref{slo0} corresponding to $v$, and let $\Bar{X}^{\varepsilon}$ be the solution of \eqref{barX} corresponding to the same $v$. Then for every $1< p<\frac{2\alpha_1\alpha_2}{\alpha_1+2\alpha_2}$, we have that
\begin{equation}
\label{XX}
\lim_{\varepsilon \to 0}\sup_{v \in U}\mathbb{E}\left(\sup_{s\in[t,T]}|X^{\varepsilon}_s-\Bar{X}^{\varepsilon}_s|^{p}\right)=0,
\end{equation}
and
\begin{equation}
|u^{\varepsilon}-\bar{u}|\leq C(\mathcal{P}) \varepsilon^{p},
\end{equation}
where the averaged controlled jump diffusion $\Bar{X}_{s}$ is defined as the solution to
\begin{equation}
\label{barX}
    d\Bar{X}_{s} = \Bar{b}(\Bar{X}_{s}, v_{s}) ds + d\textbf{L}^{\alpha_1}_{s},\quad  \Bar{X}_{t} = x \in \mathbb{R}^n,
\end{equation}
and the effective drift $\Bar{b}(x^{\prime},v)$ is given by
\begin{equation}
\Bar{b}(x^{\prime},v):= \int_{\mathbb{R}^m} b(x^{\prime},y,v) d\mu^{x^{\prime}}(y).
\end{equation}
\end{thm}

\begin{remark}
Under Assumption ($\mathbf{A_{b,L}}$), one may not extend the control set $\mathcal{U}$ to $\mathcal{U}^{ex}$, defined as the set of progressively measurable processes taking values in the extended control set ${U}^{ex}:=\{\vartheta: \mathbb{R}^{m}\rightarrow U \text{measurable}\}$ to ensure that the effective Hamiltonian \eqref{OH} is of Bellman type. We refer to \cite{BH24} for further details.
\end{remark}

\section{Convergence of nonlocal HJB equations}
In this section, we construct an effective Hamiltonian and initial data to demonstrate the convergence of the solution pair of the singularly perturbed optimal nonlocal HJB equation to the solution pair of the limiting optimal nonlocal HJB equation. Under certain appropriate assumptions, by applying It\^o's formula and taking expectations on both sides, the above value function $u^{\varepsilon}$ mentioned above is the unique solution to a nonlocal  HJB equation (\ref{HJBu}).

\subsection{Viscosity solutions}
In this subsection, we recall some preliminaries for viscosity solutions for nonlocal HJB equations.
Consider the nonlocal HJB equation
\begin{equation}\label{nHJBu}
  \left\{
   \begin{aligned}
    & \partial_{t}u(t,x)  -\left(-\Delta_{x}\right)^{\alpha_{1}/2} u + \bar{H}(x,\nabla u(t,x))  - \lambda u(t,x) = 0, \quad (t,x) \in (0,T) \times \Omega, \\
    & u(T,x) =  g(x), \quad x \in \Omega,
   \end{aligned}
   \right.
\end{equation}
Here, the discount factor $\lambda \in \mathbb{R}$, $\Omega$ is the whole space $\mathbb{R}^n$, or in a bounded connected smooth open domain $\Omega \subset \mathbb{R}^n$. If $\Omega $ is a bounded domain in $\mathbb{R}^n$, we also assume the Dirichlet boundary condition
\begin{equation}
    u(x) =  h(x), \quad \text{on } [0,T]\times\Omega^{c}.
\end{equation}
where $h$ is a continuous function.

Now we recall two equivalent definitions of viscosity solutions for nonlocal HJB equations (\ref{nHJBu}), see e.g.  Barles and Imbert  \cite{GC08}, Ciomaga \cite{AC12}, and Mou \cite{CM19}.
\begin{defn}\label{VSD1}
A upper (lower) semicontinuous function $u: [0,T]\times \mathbb{R}^n \to \mathbb{R}$ is a viscosity subsolution (supersolution) to (\ref{nHJBu}) if for any bounded test-function $\varphi \in C^{1,2}([0,T)\times\mathbb{R}^n)$ such that $u - \varphi$ attains its global maximum (minimum) at $(t,x) \in (0,T)\times\Omega$, then
\begin{equation*}
    \partial_t \varphi(t,x)  -\left(-\Delta_{x}\right)^{\alpha_{1}/2}\varphi(t,x) + H(x,\nabla_{x}\varphi(t,x))  - \lambda u(t,x)  \geq (\leq)  0.
\end{equation*}
Moreover, $\varphi$ is a viscosity solution of (\ref{nHJBu}) if it is both a subsolution and a supersolution of (\ref{nHJBu}).
\end{defn}

Note that the above definition involves the maximum and minimum of $u-\varphi$ in the whole space, and it is not convenient to use in many situations. Now we give equivalent definitions for viscosity solutions, which only rely on the maximum and minimum in bounded domains. We first introduce the following localized operators
\begin{equation}\label{IOd}
    I^{\delta}[\varphi](x)= \int_{B_{\delta}(0)} \left[\varphi(x+z)- \varphi(x)-\langle \nabla \varphi(x), z\rangle \right]\frac{C_{n,\alpha}}{|z|^{n+\alpha}}dz,\quad x\in\mathbb{R}^n,
\end{equation}
and
\begin{equation}\label{IOdc}
    I^{\delta,c}[\varphi](x)= \int_{B^{c}_{\delta}(0)} \left[\varphi(x+z)- \varphi(x) \right]\frac{C_{n,\alpha}}{|z|^{n+\alpha}}dz,\quad x\in\mathbb{R}^n,
\end{equation}
where the test function $\varphi \in C^{2}(B_{\delta}(x))$ for some constant $\delta>0$.

\begin{remark}
We remark that for a general L\'evy measure $\nu$, the localized operators are given by
\begin{equation*}
    I^{\delta}[\varphi](x)= \int_{B_{\delta}(x)} \left[\varphi(x+z)-\varphi(x)-1_{B_{1}(0)}(z)\langle \nabla \varphi(x), z\rangle \right] \nu(dz),\quad x\in\mathbb{R}^n,
\end{equation*}
and
\begin{equation*}
    I^{\delta,c}[\varphi,p](x)= \int_{B^{c}_{\delta}(x)} \left[\varphi(x+z)-\varphi(x) - 1_{B_{1}(0)} p\cdot z \right]\nu(dz),\quad x\in\mathbb{R}^n.
\end{equation*}
In our case, the L\'evy measure $\nu(dz) = C_{n,\alpha}/|z|^{n+\alpha}dz$ is symmetric. Then for every $\delta >0$, we have
\begin{align*}
    I^{\delta}[\varphi](x)= & \int_{B_{\delta}(x)} \left[\varphi(x+z)-\varphi(x)-1_{B_{1}(0)}(z)\langle \nabla \varphi(x), z\rangle \right] \frac{C_{n,\alpha}}{|z|^{n+\alpha}}dz \\
    =&  \int_{B_{\delta}(x)} \left[\varphi(x+z)- \varphi(x)-\langle \nabla \varphi(x), z\rangle \right]\frac{C_{n,\alpha}}{|z|^{n+\alpha}}dz,
\end{align*}
and
\begin{align*}
     I^{\delta,c}[\varphi,p](x)= & \int_{B^{c}_{\delta}(x)} \left[\varphi(x+z)-\varphi(x) - 1_{B_{1}(0)}(z) p\cdot z \right]\frac{C_{n,\alpha}}{|z|^{n+\alpha}}dz \\
     = & \int_{B^{c}_{\delta}(x)} \left[\varphi(x+z)- \varphi(x) \right]\frac{C_{n,\alpha}}{|z|^{n+\alpha}}dz.
\end{align*}
Thus the localized operators can be simplified as in (\ref{IOd}) and (\ref{IOdc}).
\end{remark}

\begin{defn}\label{VSD2}
A upper (lower) semicontinuous function $\varphi: [0,T]\times \mathbb{R}^n \to \mathbb{R}$ is a viscosity sub(super)-solution if for any bounded test-function $\varphi \in C^{1,2}([0,T)\times B_{\delta}(x))$ such that $(t,x)$ is a maximum (minimum) point of $u - \varphi$ on $B_{\delta}(x,t)$, then
\begin{align}
     \partial_{t}\varphi (t,x) +  \nonumber I_{x}^{\delta}[\varphi](t,x)+ I_{x}^{\delta,c}[u](t,x)  + H\left(x,\nabla_{x} \varphi(t,x)\right)  - \lambda u \geq  ( \leq) 0.
\end{align}
We say $u$ is a viscosity solution of (\ref{nHJBu}) if it is both a subsolution and a supersolution of (\ref{nHJBu}).
\end{defn}

The existence and the uniqueness of the solution $u$ of \eqref{nHJBu} on $\mathbb{R}^n$ or smooth connect bounded open set $\Omega \subset \mathbb{R}^n$ are established in the framework of the viscosity solution, see e.g. Pham \cite{HP98}, Barles and Imbert \cite{GC08}, Ciomaga \cite{AC12}. Moreover, by Schauder estimates (see e.g. \cite{GB12}), $u^{\varepsilon}$ is Lipschitz continuous with respect to $t,x$. By time-reversal transformation $t \mapsto T-t$ and \cite[Theorem 18]{AC12}, the nonlocal HJB equation \eqref{nHJBu} in bounded domain $\Omega$ also has the following maximum principle.
\begin{cor}\label{Maxp}
(Maximum principle) Let $\Omega \subset \mathbb{R}^n$ be a open, simply connected, and bounded domain. Let $u$ be a viscosity subsolution to (\ref{nHJBu}) that attains a maximum at $(x_{0},t_{0}) \in [0,T] \times \Omega$. Then $u$ is a constant in $[t_{0},T]\times\Omega$.
\end{cor}

In the whole space, we have the following comparison result for the HJB equation (\ref{nHJBu}), see e.g. \cite{GC08, AC12}.

\begin{cor}
Let the upper semicontinuous function $u:\mathbb{R}^n \times (0,T] \to \mathbb{R}$ and the lower semicontinuous function $v:\mathbb{R}^n \times (0,T] \to \mathbb{R}$ be respectively a sub and a super solution of the HJB equation (\ref{nHJBu}). Then $u \leq v$.
\end{cor}



In a non-periodic setting for the fast component, it is necessary to establish the Liouville property for viscosity solutions of nonlocal elliptic equations. Prior to presenting this Liouville property, we first introduce a Lyapunov function for the nonlocal elliptic equation associated with the fast process $Y^{\varepsilon}_{t}$, following the construction given in \cite[Section 3]{RD09}.

\begin{lem}\label{lyaf}
The function $w(y) = \sqrt{1+|y|^2}$ is a Lyapunov function for the fast component $Y^{\varepsilon}_{t}$. That is, $\lim_{|y|\to \infty}w(y) = \infty$, and for every $y\in \mathbb{R}^m$, there exists a constant $R_{0}>0$ such that
\begin{equation}
    -\mathcal{L}_{2}w(y) \geq 0,\quad \text{for}\ |y| > R_{0}.
\end{equation}
\end{lem}

The following lemma establishes the Liouville property, extending the result of \cite[Theorem 3.11]{BM16} to higher dimensions. The proof is
left into Appendix 6.

\begin{lem}\label{LiouP}
Let $V$ be a viscosity subsolution to
\begin{equation}\label{NonlocalL2}
    -\mathcal{L}_{2}  V(y) = 0, \quad y\in\mathbb{R}^m.
\end{equation}
If $V$ is bounded, then $V$ is a constant.
\end{lem}

\subsection{Effective Hamiltonian and effective initial value}

In this section, we define the effective Hamiltonian and the effective initial value via the cell problem. We introduce the following $\varepsilon$-cell problem
\begin{equation}\label{cell}
      -\mathcal{L}_{2} w_{\varepsilon}(y) +  \varepsilon w_{\varepsilon}(y)=  -H(\Bar{x},y,\Bar{p}), \quad \text{in} \   \mathbb{R}^m,
\end{equation}
whose solution $w_{\varepsilon}$ is called approximate corrector. In order to determine the initial value, we fix slow variables $\Bar{x},\Bar{p}:=\nabla_{x}u^{\varepsilon}$. The next lemma states that $-\varepsilon w_{\varepsilon}$ converges to the effective Hamiltonian $\Bar{H}$ as $\varepsilon\to 0$.
\begin{lem}\label{cellL}
For every $\Bar{x}\in\mathbb{R}^n,\Bar{p}\in\mathbb{R}^n$, and $\varepsilon \in (0,1)$, there exists a solution $w_\varepsilon \in C^{1}(\mathbb{R}^m)$ to the $\varepsilon$-cell problem (\ref{cell}) such that
\begin{equation}
    \lim_{\varepsilon\to 0}\varepsilon w_{\varepsilon}(y) :=  -\Bar{H}(y,\Bar{p}_{x}) = -\int_{\mathbb{R}^m} H(\Bar{x},y,\Bar{p}_{x}) d\mu^{\Bar{x}}(y) \quad \text{local uniformly in } \mathbb{R}^m,
\end{equation}
where $\mu^{\Bar{x}}$ is the invariant probability measure on $\mathbb{R}^m$ to $Y^{\Bar{x}}_{s}$.
\end{lem}
\begin{proof}
We denote $h(y): =  H(\Bar{x},y,\Bar{p})$. Since $h(y)$ is bounded and Lipschitz continuous, by the resolvent associated with the semigroup $P_s:= e^{s\mathcal{L}_{2}}$ (see e.g. \cite{ZX16}), the $\varepsilon$-cell problem (\ref{cell}) has a unique solution
\begin{equation*}
   w_{\varepsilon}(y) :=  -\int^{\infty}_{0} P_{s} h(y) e^{-\varepsilon s}ds.
\end{equation*}
Note that $\sup_{y \in\mathbb{R}^m}(h(y) \vee 0)$ is a subsolution to (\ref{cell}), and $\inf_{y \in\mathbb{R}^m}(h(y) \wedge 0)$ are supersolution to (\ref{cell}). Using the comparison principle with $\sup_{y \in\mathbb{R}^m}(h(y) \vee 0)$ and $\inf_{y \in\mathbb{R}^m}(h(y) \wedge 0)$, the functions $\varepsilon w_{\varepsilon}$ are uniformly bounded
\begin{equation}
    |\varepsilon w_{\varepsilon}(y)| \leq  \sup_{y \in\mathbb{R}^m}(h(y) \vee 0) - \inf_{y \in\mathbb{R}^m}(h(y) \wedge 0):= C_{h}.
\end{equation}
Since $-\mathcal{L}_{2} (\varepsilon w_{\varepsilon}) = -\varepsilon w_{\varepsilon}-\varepsilon h$, we have $|\mathcal{L}_{2} (\varepsilon w_{\varepsilon})| \leq 2\varepsilon C_{h}$.
By Schauder estimates for linear nonlocal elliptic equations (see e.g. \cite{LC09}),  the family $\{\varepsilon w_{\varepsilon}\}_{\varepsilon \in (0,1)}$ is equi $\beta$-H\"older continuous in some $B_{R}(0)$ with some $\beta \in (0,1)$ and $R>0$. Thus by the Arzel\'a-Ascoli theorem, there exists a subsequence $\varepsilon_{n} \to 0$ such that $\varepsilon_{n}w_{\varepsilon_{n}} \to v$ locally uniformly as $n \to 0$, and $-\mathcal{L}_{2}v = 0$ in $\mathbb{R}^m$. Thus, by the Liouville property (Lemma \ref{LiouP}), $v$ is constant.

Note that the solution $w_{\varepsilon}$ has stochastic representation
\begin{equation}\label{Storf}
    w_{\varepsilon}(y)= -\mathbb{E}\int^{\infty}_{0} h(Y^{\Bar{x},y}_{r}) e^{-\varepsilon r}dr.
\end{equation}
Then integrating both sides of above representation formula (\ref{Storf}) with respect to $\mu$ and using the Fubini theorem, we get
\begin{equation}
    \int_{\mathbb{R}^m} w_{\varepsilon}(y)d\mu(y) = -\int^{\infty}_{0}\int_{\mathbb{R}^m}h(y)d\mu(y)e^{-\varepsilon r}dr = -\frac{1}{\varepsilon} \int_{\mathbb{R}^m}h(y)d\mu(y).
\end{equation}
Thus for every convergence subsequence $\varepsilon_{n}w_{\varepsilon_{n}}$, there exists a unique constant limit $v$, so that
\begin{equation}
    \lim_{n\to \infty}\varepsilon_{n}w_{\varepsilon_{n}} = v = -\int_{\mathbb{R}^m}h(y)d\mu(y):= -\Bar{H}(\cdot,\Bar{p}_{x}), \quad \text{local uniformly in } \mathbb{R}^m.
\end{equation}
\end{proof}

To study the effective initial data, we introduce the Cauchy cell problem
\begin{equation}\label{cellI}
 \left\{
 \begin{aligned}
   &  \partial_{s}w_{x} =  \mathcal{L}_{2}w_{x}, & \text{in} \  (0,\infty)\times \mathbb{R}^m,\\
   &  w_{x}(0,y) = g(x,y),  & \text{in} \   \mathbb{R}^m,
   \end{aligned}
   \right.
\end{equation}
In the next lemma, we show that the effective initial data is given by the following Cauchy cell problem.

\begin{lem}\label{cellT}
Under assumptions $(\bf{A_{b}})$, $(\bf{A_{c}})$, $(\bf A_{L})$, $(\bf A_{g})$ in Section 2, for every fixed $x\in \mathbb{R}^n$, the Cauchy problem (\ref{cellI}) has a unique classical solution $w$, and
\begin{equation}
    \lim_{s \to \infty} w_{x}(s, y) = \int_{\mathbb{R}^m} g(x,y)\mu^{x}(dy) = \Bar{g}(x) \quad \text{local uniformly in } y,
\end{equation}
where $\mu^{x}$ denotes the ergodic measure on $\mathbb{R}^m$ associated with the frozen process $Y^{x,y}_{s}$.
\end{lem}
\begin{proof}
By semigroup estimates associated with generator $\mathcal{L}_{2}$ (see e.g. \cite{ZX16}), the Cauchy cell problem (\ref{cellI}) has a unique classical solution $w_{x}$ with stochastic representation
\begin{equation}
    w_{x}(s,y)= P^{x}_{s} g(x,y) = \mathbb{E}g(x,Y^{x,y}_{s}).
\end{equation}
Using the exponential ergodicity (Proposition \ref{ergodic11}),
\begin{equation*}
    \lim_{s \to \infty} w_{x}(s, y) = \lim_{s \to \infty} \mathbb{E}g(x,Y^{x,y}_{s}) = \int_{\mathbb{R}^m} g(x,y)\mu^{x}(dy) = \Bar{g}(x), \quad \text{local uniformly in } y.
\end{equation*}
\end{proof}

To study the Lipschitz continuity of the effective Hamiltonian \(\bar{H}\), we need the following estimates for the jump diffusion \(Y^{x,y}_{s}\) from \cite[Lemma 4]{YZ24}.

\begin{prop}
\label{ESYxy}
Suppose that assumptions $(\bf{A_{b}})$, $(\bf{A_{c}})$ hold. Then for every $s>0$, $x_{1},x_{2}\in\mathbb{R}^n$, $y_{1},y_{2}\in \mathbb{R}^m$, we have
\begin{equation}
\label{dY}
    \left|Y^{x_1,y_1}_s-Y^{x_2,y_2}_s\right|^2 \leq e^{-\frac{\beta s}{2}}|y_1-y_2|^2 +C(\|c_1\|,\beta)|x_1-x_2|^2.
\end{equation}
where $C$, $\beta$ is a positive constant independent of $s$.
\end{prop}
\begin{proof}
By the equation \eqref{0nsde}, we have
\begin{equation*}
d\left( Y^{x_{1},y_{1}}_{s} - Y^{x_{2},y_{2}}_{s} \right)=\left[c(x_{1}, Y^{x_1,y_1}_s)-c(x_{2},Y^{x_2,y_2}_s)\right] dt, \quad Y^{x_1,y_1}_t -Y^{x_2,y_2}_t =y_1-y_2.
\end{equation*}
Multiplying both sides by $2\left(Y^{x_1,y_1}_s-Y^{x_2,y_2}_s\right)$, by Assumption ($\bf{A_{c}}$) and Young's inequality, we have
\begin{equation*}
\begin{aligned}
\frac{d}{ds}\left|Y^{x_1,y_1}_s-Y^{x_2,y_2}_s \right|^2&=2\left\langle c(x_{1}, Y^{x_1,y_1}_s)-c(x_{2},Y^{x_2,y_2}_s), Y^{x_1,y_1}_s-Y^{x_2,y_2}_s \right\rangle \\
& \leq 2\left\langle c(x_{1}, Y^{x_1,y_1}_s)-c(x_{1},Y^{x_2,y_2}_s), Y^{x_1,y_1}_s-Y^{x_2,y_2}_s \right\rangle \\
&\quad +2\left\langle c(x_{1}, Y^{x_2,y_2}_s)-c(x_{2},Y^{x_2,y_2}_s), Y^{x_1,y_1}_s-Y^{x_2,y_2}_s \right\rangle \\
& \leq -2\beta\left|Y^{x_1,y_1}_s-Y^{x_2,y_2}_s\right|^2+C(\|c_1\|,\beta)|x_1-x_2|\left|Y^{x_1,y_1}_s-Y^{x_2,y_2}_s\right|\\
& \leq - \beta \left|Y^{x_1,y_1}_s-Y^{x_2,y_2}_s\right|^2+ C(\|c_1\|,\beta)|x_1-x_2|^2.
\end{aligned}
\end{equation*}
Hence, the Gr\"{o}nwall inequality yields that
\begin{equation*}
\left|Y^{x_1,y_1}_s-Y^{x_2,y_2}_s\right|^2 \leq e^{-\beta s}|y_1-y_2|^2 +C(\|c_1\|,\beta)|x_1-x_2|^2.
\end{equation*}
\end{proof}

In the next lemma, we show the Lipschitz continuity of the effect Hamiltonian $\Bar{H}$ and the effect terminal data $\Bar{g}$.
\begin{lem}
Under assumptions $(\bf{A_{b}})$, $(\bf{A_{c}})$, $(\bf A_{L})$, $(\bf A_{g})$, the effect Hamiltonian $\Bar{H}: \mathbb{R}^m\times\mathbb{R}^m \to \mathbb{R}$ and the effect terminal data $\Bar{g}$ are Lipschitz continuous with respect to $x$ and $p$.
\end{lem}
\begin{proof}
For every $(x_{1},p_{1}),(x_{2},p_{2}) \in \mathbb{R}^m\times\mathbb{R}^m $, we have
\begin{align}
    \bar{H}(x_{1},p_{1}) - \bar{H}(x_{2},p_{2}) = \left[\bar{H}(x_{1},p_{1}) - \bar{H}(x_{1},p_{2})\right] + \left[\bar{H}(x_{1},p_{2}) - \bar{H}(x_{2},p_{2})\right].
\end{align}
Recall the definition of effect Hamiltonian  $\bar{H}(x,p)=\int_{\mathbb{R}^m} H(x,y, p) d\mu^{x}(dy)$, where $\mu^{x}$ is the ergodic measure to the fast component $Y^{x,y}_{s}$. By Lipschitz continuity of the Hamiltonian $H(x',y,p)$, we obtain
\begin{align}
    \left|\bar{H}(x_{1},p_{1}) - \bar{H}(x_{1},p_{2})\right| \leq & \int_{\mathbb{R}^m} |H(x_{1},y, p_{1}) - H(x_{1},y, p_{2})| d\mu^{x_{1}}(dy)\leq  C|p_{1}-p_{2}|.
\end{align}
By Proposition \ref{ESYxy} and Lipschitz continuity of $H(x',y',p)$, we get for every $T>0$, $y_{1},y_{2} \in \mathbb{R}^{m}$, and some $\beta'>0$,
\begin{equation}
\begin{aligned}
    |\bar{H}(x_{1},p_{2}) - \bar{H}(x_{2},p_{2})| \leq & \left|\bar{H}(x_{1},p_{2}) - \mathbb{E} H(x_{1},Y^{x_{1},y_{1}}_{s},p_{2}) \right| \\&+ \left|\mathbb{E} H(x_{1},Y^{x_{1},y_1}_{s},p_{2}) - \mathbb{E} H(x_{2},Y^{x_{2},y_{2}}_{s},p_{2}) \right| \nonumber\\
    & + \left|\mathbb{E} H(x_{2},Y^{x_{2},y_{2}}_{s},p_{2}) - \bar{H}(x_{2},p_{2})\right| \nonumber \\
    \leq & \left|\int_{\mathbb{R}^m} H(x_{1},y,p_{2}) \mu^{x_{1}}(dy) - \mathbb{E} H(x_{1},Y^{x_{1}}_{s},p_{2}) \right| \nonumber \\
    & + C\left[\mathbb{E}|Y^{x_{1},y_{1}}_{s} - Y^{x_{2},y_{2}}_{s}| + |x_{1}-x_{2}|\right] \nonumber \\
    & + \left| \int_{\mathbb{R}^m} H(x_{2},y,p_{2})\mu^{x_{2}}(dy) - \mathbb{E} H(x_{2},Y^{x_{2}}_{s},p_{2})\right| \nonumber \\
    \leq & e^{-\beta s}|y_{1}-y_{2}| + C|x_{1}-x_{2}|.
\end{aligned}
\end{equation}
Letting $s\to \infty$, we arrive at the Lipschitz continuity of the definition of effect Hamiltonian $\bar{H}$. The Lipschitz  continuity of the terminal data $\Bar{g}$ can be proved using a similar argument.

\end{proof}
We now establish the uniform estimate of the solution $\{X^{\varepsilon}_t\}_{t\geq 0}$.
\begin{lem}
\label{slo}
Suppose the assumptions in Theorem \ref{theorem1} hold. Let $v$ be any admissible control, let $(X^{\varepsilon},Y^{\varepsilon})$ be the solution
of \eqref{slo0} corresponding to $v$. Then for any $\varepsilon > 0$, there exists a unique strong solution $\{(X^{\varepsilon}_s,Y^{\varepsilon}_s)\}_{t \leq s \leq T}$ to system \eqref{slo0}. Moreover, for any $p\in [1, \alpha_1)$, there exist constants $C(p,T, K_2)$ such that
\begin{equation}
\sup_{v\in U}\sup_{\varepsilon \in (0,1)}\mathbb{E}\left(\sup_{s\leq [t,T]}|X^{\varepsilon}_s|^{p}\right) \leq C(p,T,K_2)(1+|x|^p).
\end{equation}
\end{lem}
\begin{proof}
By Burkholder-Davis-Gundy's inequality, it is easy to know
\begin{equation}
\sup_{v\in U}\sup_{\varepsilon \in (0,1)}\mathbb{E}\left(\sup_{s\leq [t,T]}|X^{\varepsilon}_s|^{p}\right) \leq C(p)|x|^{p}+C(p,T, K_2)\int^{T}_{t}\sup_{v\in U}\sup_{\varepsilon \in (0,1)}\mathbb{E}\left(\sup_{s\leq [t,T]}|X^{\varepsilon}_r|^{p}\right)dr+C(p,T).
\end{equation}
Thus, Gr\"onwall inequality yields
\begin{equation}
\sup_{v\in U}\sup_{\varepsilon \in (0,1)}\mathbb{E}\left(\sup_{s\leq [t,T]}|X^{\varepsilon}_s|^{p}\right) \leq C(p,T,K_2)(1+|x|^p).
\end{equation}
\end{proof}

\subsection{The convergence result}

Now we prove our main convergence result in theorem \ref{theorem1}. Motivated by the argument for second order differential operator cases from \cite{BCM10}, the proof of Theorem \ref{theorem1} is based on relaxed semilimits, Liouville property, and perturbed test function methods.

\begin{proof}(Proof of Theorem \ref{theorem1})
The proof is divided into five steps.

\textbf{Step 1} (relaxed semilimits).
Since the solutions $u^{\varepsilon}$ are locally equi-bounded in $\mathbb{R}^{+}\times \mathbb{R}^{n}$, uniformly in $\varepsilon$, we define the relaxed semilimits as
\begin{equation}
    \Bar{u}(t,x,y) = \limsup_{t'\to t, x'\to x, y'\to y, \varepsilon\to 0} u^{\varepsilon}(t,x,y), \quad \underline{u}(t,x,y) = \liminf_{t'\to t, x'\to x, y'\to y, \varepsilon\to 0} u^{\varepsilon}(t,x,y)
\end{equation}
for $(t,x,y) \in (0,T)\times \mathbb{R}^n \times \mathbb{R}^m$. The terminal value is given by
\begin{equation}
    \Bar{u}(T,x,y) = \limsup_{t'\to T, x'\to x, y'\to y} \Bar{u}(t,x,y), \quad \underline{u}(T,x,y) = \liminf_{t'\to T, x'\to x, y'\to y}  \underline{u}(t,x,y).
\end{equation}
By assumptions$(\bf B_{L})$, $(\bf B_{g})$, together with the moment estimate in Lemma \ref{slo}, we get that there exists a constant $K>0$ independent with $(t,x,y)$ and $\varepsilon$, such that for every $p \in [1,\alpha_{1})$,
\begin{equation}\label{uppueE}
    |u^{\varepsilon}(t,x,y)| = \left|\sup_{v \in {U}}J^{\varepsilon}(v,x,y,t) \right| \leq K(1+|x|^p).
\end{equation}
Since $u^{\varepsilon}$ is $p$th order polynomial growth with respect to $x$, the relaxed semi-limits $\Bar{u}$ and $\underline{u}$ are also $p$th order polynomial growth with respect to $x$, i.e. for every $(t,x,y) \in [0,T]\times \mathbb{R}^n \times\mathbb{R}^m$ and $p \in [1,\alpha_{1})$,
\begin{equation}\label{boundu}
    |\Bar{u}(t,x,y)| \leq K(1+|x|^p),\quad |\underline{u}(t,x,y)| \leq K(1+|x|^p).
\end{equation}

\textbf{Step 2} ($\Bar{u}$ and $\underline{u}$ do not depend on $y$). Now we show that $\Bar{u}$ and $\underline{u}$ do not depend on $y$ for every $(t,x)\in [0,T)\times\mathbb{R}^n$ by Liouville property. We only prove the claim for $\Bar{u}$, since the proof for $\underline{u}$ is completely analogous.
We first show that for every fixed $\Bar{t},\Bar{x} \in (0,T)\times\mathbb{R}^n$, the function $y \mapsto \Bar{u}(\Bar{t},\Bar{x},y)$ is a subsolution to (\ref{NonlocalL2}).
For every fixed $\bar{y} \in \mathbb{R}^m$, let $\phi \in C^{2}(\mathbb{R}^m)$ be a test function such that, for some $\delta \in (0,1)$, the function $u(\bar{t}, \bar{x}, \cdot) - \phi(\cdot)$ attains a strict local maximum at $\bar{y}$ in the ball $\bar{B}_{\delta}(\bar{y})$, and $\phi$ satisfies $1 \leq \phi < K'$ in $\bar{B}_{\delta}(\bar{y})$ for some constant $K' > 1$.
Similar with \cite[Chapter II, Lemma 1.17]{MC97}), we introduce the test function $\phi^{\varepsilon}(t,x,y) = \phi(y)\left(1 + \frac{|x-\Bar{x}|^2 + |t-\Bar{t}|^2}{\varepsilon}\right)$ for every $\varepsilon >0$. Since the relaxed semi-limit $\bar{u}$ has bound \eqref{boundu}, $\Bar{u}-\phi_{\varepsilon}$ has the local maximum point $(t_{\varepsilon}, x_{\varepsilon}, y_{\varepsilon}) \in\Bar{B}_{\delta}(\Bar{t}, \Bar{x}, \Bar{y})$, so that
\begin{align}\label{epsEc}
    \Bar{u}(t_{\varepsilon}, x_{\varepsilon}, y_{\varepsilon}) - \phi^{\varepsilon}(t_{\varepsilon}, x_{\varepsilon}, y_{\varepsilon}) = & \Bar{u}(t_{\varepsilon}, x_{\varepsilon}, y_{\varepsilon}) - \phi(y_{\varepsilon})\left(1 + \frac{|x_{\varepsilon}-\Bar{x}|^2 + |t_{\varepsilon}-\Bar{t}|^2}{\varepsilon}\right) \nonumber \\
    \geq & \Bar{u}(\bar{t}, \bar{x}, \bar{y}) - \phi(\bar{y}).
\end{align}
Since $K' > \phi \geq 1$ in $\Bar{B}_{\delta}(\Bar{y})$, we have
\begin{equation*}
     \frac{|x_{\varepsilon}-\Bar{x}|^2 + |t_{\varepsilon}-\Bar{t}|^2}{\varepsilon} \leq   \Bar{u}(t_{\varepsilon}, x_{\varepsilon}, y_{\varepsilon}) - \Bar{u}(\bar{t}, \bar{x}, \bar{y}) + \phi(\bar{y})< C
\end{equation*}
for some constant $C>0$ and $(t,x) \in \bar{B}_{\delta}(\bar{t},\bar{x})$.
Therefore, there exists a convergence subsequences $$\{(t_{\varepsilon_{n}},x_{\varepsilon_{n}}, y_{\varepsilon_{n}})\}_{n=1}^{\infty} \subset \Bar{B}_{\delta}(\bar{t}, \bar{x}, \bar{y}),$$ so that
\begin{align*}
    & (t_{\varepsilon_{n}},x_{\varepsilon_{n}}, y_{\varepsilon_{n}}) \to  (t_{1}, x_{1}, y_{1}) \in \Bar{B}_{\delta}(\bar{t}, \bar{x}, \bar{y}), \\
    & M_{n}:=\left(1 + \frac{|x_{\varepsilon_{n}}-\Bar{x}|^2 + |t_{\varepsilon_{n}}-\Bar{t}|^2}{\varepsilon_{n}}\right) \to M \geq 1.
\end{align*}
Then letting $\varepsilon_{n} \to 0$ in \eqref{epsEc}, we have
\begin{equation*}
    \Bar{u}(t_{1}, x_{1}, y_{1}) - \phi(y_{1}) \geq \Bar{u}(t_{1}, x_{1}, y_{1}) - M\phi(y_{1}) \geq \Bar{u}(\bar{t},\bar{x},\bar{y}) - \phi(\bar{y}) .
\end{equation*}
Since $u(\Bar{t},\Bar{x},\cdot) - \phi(\cdot) $ has a strict local maximum at $\Bar{y}$ in $B_{\delta}(\Bar{y})$, and $\phi \geq 1$ in $\Bar{B}_{\delta}(\Bar{y})$, we get $(t_{1},x_{1}, y_{1}) = (\bar{t}, \bar{x}, \bar{y})$, and
\begin{equation*}
    (t_{\varepsilon_{n}},x_{\varepsilon_{n}}, y_{\varepsilon_{n}}) \to  (\bar{t}, \bar{x}, \bar{y})\ \text{as} \ \varepsilon_{n} \to 0.
\end{equation*}
Since $u^{\varepsilon}$ is a subsolution to (\ref{HJBu}), we have
\begin{align*}
     & \partial_t \phi^{\varepsilon_{n}}(t_{\varepsilon_{n}},x_{\varepsilon_{n}},y_{\varepsilon_{n}}) +  \frac{1}{\varepsilon_{n}}\left[I_{y}^{\delta}[\phi_{\varepsilon_{n}}](x_{\varepsilon_{n}},y_{\varepsilon_{n}})+ I_{y}^{\delta,c}[u^{\varepsilon_{n}}](x_{\varepsilon_{n}},y_{\varepsilon_{n}})+ \nabla_{y} \phi_{\varepsilon_{n}}(x_{\varepsilon_{n}},y_{\varepsilon_{n}})\right] \nonumber \\
     & \quad  +I_{x}^{\delta}[\phi_{\varepsilon_{n}}](x_{n},y_{n})+ I_{x}^{\delta,c}[u^{\varepsilon_{n}}](x_{\varepsilon_{n}},y_{\varepsilon_{n}}) + H\left(x_{\varepsilon_{n}},y_{\varepsilon_{n}}\nabla_{x} \phi_{\varepsilon_{n}} \right) - \lambda u^{\varepsilon_{n}}(t_{\varepsilon_{n}},x_{\varepsilon_{n}},y_{\varepsilon_{n}}) \leq 0
\end{align*}
It follows that
\begin{align*}
     & I_{y}^{\delta}[\phi_{\varepsilon_{n}}](x_{\varepsilon_{n}},y_{\varepsilon_{n}})+ I_{y}^{\delta,c}[u^{\varepsilon_{n}}](x_{\varepsilon_{n}},y_{\varepsilon_{n}})+ c(x_{\varepsilon_{n}},y_{\varepsilon_{n}}) \cdot \nabla_{y} \phi_{\varepsilon_{n}}(x_{\varepsilon_{n}},y_{\varepsilon_{n}}) \nonumber \\
    \leq & \varepsilon_{n} \big[-\partial_t \phi^{\varepsilon_{n}}(t_{\varepsilon_n},x_{\varepsilon_{n}},y_{\varepsilon_{n}}) - I_{x}^{\delta}[\phi_{\varepsilon_{n}}](x_{\varepsilon_{n}},y_{\varepsilon_{n}})- I_{x}^{\delta,c}[u^{\varepsilon_{n}}](t_{\varepsilon_{n}},x_{\varepsilon_{n}},y_{\varepsilon_{n}}) \nonumber\\
    & \quad \quad   -H\left(x_{\varepsilon_{n}},y_{\varepsilon_{n}},\nabla_{x} \phi_{\varepsilon_{n}}\right) + \lambda u^{\varepsilon_{n}}(t_{\varepsilon_{n}},x_{\varepsilon_{n}},y_{\varepsilon_{n}}) \big] .
\end{align*}
Note that the term in square bracket is uniformly bounded for $\varepsilon_n$. Letting $\varepsilon_{n} \to 0$, we get
\begin{equation}
    I^{\delta}_{y}[\phi](\Bar{t},\Bar{x},\Bar{y}) + I^{\delta,c}_{y}[\Bar{u}](\Bar{t},\Bar{x},\Bar{y}) + c(\Bar{x},\Bar{y})\cdot\nabla \phi(\Bar{y}) \leq  0.
\end{equation}
Thus $\Bar{u}$ is a viscosity subsolution to (\ref{NonlocalL2}). Since $\Bar{u}$ is bounded in $y$ according to estimate \eqref{boundu}, by Liouville property (Lemma \ref{LiouP}), $\Bar{u}$ do not depend on $y$.

\textbf{Step 3} ($\Bar{u}$ and $\underline{u}$ are subsolution and supersolution of the limit PDE). We claim that $\Bar{u}$ and $\underline{u}$ are sub- and supersolutions of the effective HJB equation (\ref{0EHJBu}) on $(0,T)\times \mathbb{R}^n$. We only show that $\Bar{u}$ is a subsolution to (\ref{0EHJBu}) by the perturbed test function method and contradiction argument.
The proof that $\underline{u}$ is a supersolution is analogous.

For every fixed $\Bar{t},\Bar{x} \in (0,T)\times\mathbb{R}^n$, assume that there exists a test function $\varphi \in C^{2}((0,T)\times \mathbb{R}^n)$ so that $\Bar{u} - \varphi$ has a global maximum point $(\Bar{t},\Bar{x})\in (0,T)\times \mathbb{R}^n$, $\Bar{u}(\Bar{t},\Bar{x})= \varphi(\Bar{t},\Bar{x})$, and
\begin{align}
   \partial_{t}\varphi (\Bar{t},\Bar{x})  -(-\Delta)^{\alpha_{1}/2} \varphi (\Bar{t},\Bar{x})+ \Bar{H}\left(\Bar{x},\nabla_{x} \varphi (\Bar{t},\Bar{x}) \right) - \lambda \Bar{u}(\Bar{t},\Bar{x}) < -4\gamma
\end{align}
for some small constant $\gamma >0$. By continuity of $H$ and $\Bar{H}$, and regularity of $\varphi$, we can choose $r > 0$ small enough, so that for every $(t,x) \in B_{r}(\Bar{t},\Bar{x})$ and $y\in \mathbb{R}^m$,
\begin{equation}\label{Ec41}
    \partial_{t}\varphi (t,x) -(-\Delta)^{\alpha_{1}/2} \varphi (t,x)+ \Bar{H}\left(x,\nabla_{x} \varphi (t,x) \right) - \lambda \varphi (t,x) < -3\gamma,
\end{equation}
and
\begin{equation}\label{Ec42}
    |\Bar{H}\left(x,\nabla_{x} \varphi (t,x) \right) -\Bar{H}\left(x,\nabla_{x} \varphi (\Bar{t},\Bar{x}) \right)|+ |H(x,y,\nabla_{x} \varphi (t,x)) - H(\Bar{x},y,\nabla_{x} \varphi (\Bar{t},\Bar{x}))| < \gamma.
\end{equation}
Now we denote $\Bar{p} = \nabla_{x} \varphi (\Bar{t},\Bar{x})$ and introduce the perturbed test function
\begin{equation}
    \varphi^{\varepsilon}(t,x,y):= \varphi(t,x) - \varepsilon w_{\varepsilon}(y),
\end{equation}
where $w_{\varepsilon}$ is the approximate corrector which is given by the $\varepsilon$-cell problem (\ref{cell}). By Lemma \ref{cellL} and the H\"{o}lder continuity of $H$ and $\varphi$, for every fixed $R>0$ we can choose $\varepsilon$ small enough, so that
\begin{equation}\label{Ec43}
    |\varepsilon w_{\varepsilon}(y) + \Bar{H}(\Bar{x},\Bar{p})| < \gamma, \quad \forall y\in B_{R}(0).
\end{equation}
Then combining with (\ref{Ec41}),(\ref{Ec42}), and (\ref{Ec43}), we obtain that for every $(t,x,y)\in B_{r}(\Bar{t},\Bar{x})\times B_{R}(\Bar{y})$,
\begin{equation}
\begin{aligned}
    & \partial_t \varphi^{\varepsilon}(t,x,y) -(-\Delta_{x})^{\alpha_{1}/2}\varphi^{\varepsilon}(t,x,y) + \frac{1}{\varepsilon}\mathcal{L}_{2}\varphi^{\varepsilon}(t,x,y) + H(x,y,\nabla_{x}\varphi^{\varepsilon}(t,x,y)) - \lambda \varphi^{\varepsilon} (t,x,y) \nonumber \\
    = & \partial_{t}\varphi (t,x) - (-\Delta)^{\alpha_1/2} \varphi (t,x)  - H\left(\bar{x},y,\bar{p}\right)  -  \varepsilon w_{\varepsilon}(y)+H(x,y,\nabla_{x}\varphi(t,x)) - \lambda \varphi^{\varepsilon} (t,x,y) \nonumber\\
    \leq & \partial_{t}\varphi (t,x) - (-\Delta)^{\alpha/2} \varphi (t,x) + \Bar{H}\left(x,\nabla_{x} \varphi (t,x) \right)  - \lambda \varphi (t,x) + |\Bar{H}\left(\bar{x},\Bar{p} \right) + \varepsilon w_{\varepsilon}(y)| \nonumber \\
     & \quad  + |H(x,y,\nabla_{x} \varphi (t,x)) - H(\Bar{x},y,\bar{p})| +|\Bar{H}\left(x,\nabla_{x} \varphi (t,x) \right) -\bar{H}\left(\bar{x},\bar{p} \right)| +  \lambda \varepsilon w_{\varepsilon}(y) \nonumber \\
    < &  -\gamma + |\lambda \varepsilon w_{\varepsilon}(y)|.
\end{aligned}
\end{equation}
Now we choose $R,\varepsilon >0$ small enough, so that $|\lambda \varepsilon w_{\varepsilon}(y)| < \gamma$ for every $y\in B_{R}(\Bar{y})$.
Then for every $(t,x,y)\in B_{r}(\Bar{t},\Bar{x})\times B_{R}(\Bar{y})$ and $\varepsilon $ small enough, the test function $\varphi^{\varepsilon}$ satisfies
\begin{equation*}
     \partial_t \varphi^{\varepsilon} -(-\Delta_{x})^{\alpha_{1}/2}\varphi^{\varepsilon} + \frac{1}{\varepsilon}\mathcal{L}_{2}\varphi^{\varepsilon} + H(x,y,\nabla_{x}\varphi^{\varepsilon}) - \lambda \varphi^{\varepsilon} < 0.
\end{equation*}
Since $H(x,y,p_1) + H(x,y,p_2) \geq H(x,y,p_1+p_2)$, $v^{\varepsilon}:=u^{\varepsilon}-\varphi^{\varepsilon}$ is a viscosity subsolution to the HJB equation
\begin{equation*}
    \left\{
 \begin{aligned}
   &  \partial_{t} V -(-\Delta_{x})^{\alpha_{1}/2}V + \frac{1}{\varepsilon}\mathcal{L}_{2}v^{\varepsilon} + H(x,y,\nabla_{x}V) - \lambda V = 0, & \text{in} \  (\bar{t},\bar{t}+r)\times B_{r}(\bar{x}) \times B_{R}(\Bar{y}),\\
   &  V = \bar{M},  & \text{in} \  (\bar{t},\bar{t}+r)\times B^c_{r}(\Bar{x}) \times B^c_{R}(\Bar{y}), \\
   &  V(\bar{t}) = \max_{(x,y)\in (B_{r}(\bar{x})\times B_{R}(\bar{y}))^c}(u^{\varepsilon}(\Bar{t},x,y) -\varphi^{\varepsilon}(\bar{t},x,y)),  & \text{in} \   \mathbb{R}^n\times\mathbb{R}^m,
   \end{aligned}
   \right.
\end{equation*}
where $\bar{M} := \max_{(t,x,y)\in ((\Bar{t},\Bar{t}+r]\times B_{r}(\Bar{x})\times B_{R}(\Bar{y}))^c}(u^{\varepsilon}(t,x,y) -\varphi^{\varepsilon}(t,x,y))$. By strong maximum principle (Proposition \ref{Maxp}), for every $(t,x,y) \in (\bar{t},\bar{t}+r]\times B_{r}(\Bar{x})\times B_{R}(\Bar{y})$ we have
\begin{equation*}
    u^{\varepsilon}(t,x,y)-\varphi^{\varepsilon}(t,x,y) \leq \max_{(t,x,y)\in ([\bar{t},\bar{t}+r]\times B_{r}(\Bar{x})\times B_{R}(\Bar{y}))^c}(u^{\varepsilon}(t,x,y) -\varphi^{\varepsilon}(t,x,y)).
\end{equation*}
Since $[\bar{t},\bar{t}+r/2]\times \Bar{B}_{r/2}(\Bar{x})\times \Bar{B}_{R/2}(\Bar{y})$ is bounded, we can choose a convergent sequences
$$\{(t_{n},x_{n},y_{n})\}_{n \geq 1} \subset [\bar{t},\bar{t}+r/2]\times \Bar{B}_{r/2}(\Bar{x})\times \Bar{B}_{R/2}(\Bar{y}),$$ so that $(t_{n},x_{n},y_{n}) \to (t_{2},x_{2},y_{2}) \in [\Bar{t},\Bar{t}+r/2]\times \Bar{B}_{r/2}(\Bar{x})\times \Bar{B}_{R/2}(\Bar{y})$ as $n\to \infty$. Then we let $n \to \infty$, and obtain
\begin{equation*}
    (\Bar{u}-\varphi)(t_{2},x_{2}) \leq \max_{(B_{r}(\Bar{t})\times B_{R}(\Bar{y}))^c}(\Bar{u}-\varphi),
\end{equation*}
which is contradicts with definition \ref{VSD1}. Thus, $\Bar{u}$ is a subsolution of the effective HJB equation.

\textbf{Step 4} (behavior of $\Bar{u}$ and $\underline{u}$ at time $T$). We show that for every $x\in \mathbb{R}^n$, $\Bar{u}(x,T) \leq \Bar{g}(x)$ and $\underline{u}(x,T) \geq \Bar{g}(x)$ by strong maximum principle. We only show the subsolution $\Bar{u}$.

For every fixed $\Bar{x}\in \mathbb{R}^n$, let $w^{r}_{\Bar{x}}$ be the viscosity solution of the Cauchy problem
\begin{equation}\label{cell2}
 \left\{
 \begin{aligned}
   &  \partial_{t}w^{r}_{\Bar{x}} =  \mathcal{L}_{2}w^{r}_{\Bar{x}}, & \text{in} \  (0,\infty)\times \mathbb{R}^m,\\
   &  w^{r}_{\Bar{x}}(0,y) = \sup_{x\in B_{r}(\Bar{x})}g(x,y),  & \text{in} \   \mathbb{R}^m,
   \end{aligned}
   \right.
\end{equation}
Since $\sup_{x\in B_{r}(\Bar{x})}g(x,y) \to g(x,y)$ as $r \to 0$, the stability of \eqref{cell2} (see e.g. \cite{GC08}) implies that $\lim_{r\to 0}w_{\bar{x}}^{r}(t,y) = w_{\bar{x}}(t,y)$ uniformly on some compact set $K \subset [0,T]\times \mathbb{R}^m$. By Lemma \ref{cellT}, for every $\gamma>0$, there exist $r>0$ small enough, and $t_{0}>0$ and $R>0$ large enough, so that for every $|y| \leq R$, and $t > t_{0}$,
\begin{equation}\label{E:w-g}
    |w^{r}_{\Bar{x}}(t,y) - \bar{g}(\Bar{x})| \leq \gamma.
\end{equation}
By regularity of $u^{\varepsilon}$, there exists a positive constant $M$ so that $u^{\varepsilon}(t,x,y) \leq M$ for every $t> t_0$ $\varepsilon \in (0,1)$, $y\in \mathbb{R}^m$, and $x\in B_{r}(\Bar{x})$.
Let $ \psi$ be a smooth nonnegative function so that $\psi(\Bar{x}) = 0$, $|-(-\Delta)^{\alpha_{1}/2}\psi|$ is uniformly bounded on $\mathbb{R}^n$, and $\psi(x) + \inf_{y\in \mathbb{R}^m}g(x,y) \geq M$ for each $x\in B^{c}_{r}(\Bar{x})$.
Then by estimate \eqref{uppueE}, there exist positive constant $C_{r} > 0$ independent with $\varepsilon$, such that for every $x\in \Bar{B}_{r}(\Bar{x}), y\in \mathbb{R}^m$,
\begin{equation}\label{eq:Cr}
    |-(-\Delta_{x})^{\alpha_{1}/2}\psi(x) +  H(x,y,\nabla\psi(x))  - \lambda u^{\varepsilon}(t,x,y)| < C_{r}.
\end{equation}
Now we introduce the test function
\begin{equation}
    \psi^{\varepsilon}(t,x,y) = w_{\Bar{x}}^{r}\left(\frac{T-t}{\varepsilon}, y \right) + \psi(x) + C_{r}(T-t).
\end{equation}
Then for $(t,x,y)\in (T-r,T)\times B_{r}(\Bar{x}) \times \mathbb{R}^m$, by \eqref{cell2} and \eqref{eq:Cr}, $\psi^{\varepsilon}$ satisfies
\begin{align}\label{Sp1}
    & \partial_{t}\psi^{\varepsilon}  -(-\Delta)^{\alpha/2}_{x}\psi^{\varepsilon} + \frac{1}{\varepsilon}\mathcal{L}_{2}\psi^{\varepsilon} + H(x,y,\nabla_{x}\psi^{\varepsilon}) - \lambda u^{\varepsilon} \nonumber \\
    = & \frac{1}{\varepsilon}\left(-\partial_{t}w^{r}_{\Bar{x}} + \mathcal{L}_{2}w^{r}_{\Bar{x}}\right) - C_{r}  -(-\Delta)^{\alpha/2}_{x}\psi + H(x,y,\nabla_{x}\psi) - \lambda u^{\varepsilon} \nonumber\\
    \leq & 0.
\end{align}
Note that the constant $ \inf_{y\in\mathbb{R}^m}\sup_{x\in B_{r}(\bar{x})}g(x,y)$ is a subsolution of (\ref{cell2}). The strong maximum principle implies that
\begin{equation}
    w^{r}_{\bar{x}}(t,y) \geq \inf_{y\in\mathbb{R}^m}\sup_{x\in B_{r}(\bar{x})}g(x,y), \quad (t,y) \in [0,\infty)\times \mathbb{R}^m.
\end{equation}
Thus for $(t,x,y)\in (T-r,T)\times B^{c}_{r}(\Bar{x}) \times \mathbb{R}^m$, we have
\begin{equation}\label{Sp2}
    \psi^{\varepsilon}(t,x,y) \geq \inf_{y\in\mathbb{R}^m}\sup_{x\in B_{r}(\Bar{x})}g(x,y) + M - \inf_{y\in\mathbb{R}^m}g(x,y) + C_{r}(T-t) \geq M.
\end{equation}
Moreover, at time $T$ we have
\begin{equation}\label{Sp3}
    \psi^{\varepsilon}(T,x,y) \geq \sup_{x\in B_{r}(\Bar{x})}g(x,y) + \psi(x) \geq g(x,y).
\end{equation}
Combining with (\ref{Sp1}),(\ref{Sp2}), and (\ref{Sp3}), we conclude that $\psi^{\varepsilon}$ is a supersolution of the parabolic equation
\begin{equation}\label{eq:TV-HJB}
    \left\{
 \begin{aligned}
   &  \partial_{t} V -(-\Delta_{x})^{\alpha_{1}/2}V + \frac{1}{\varepsilon}\mathcal{L}_{2}V + H(x,y,\nabla_{x}V) - \lambda V = 0, & \text{in} \  (T-r,T)\times B_{r}(\bar{x}) \times \mathbb{R}^m,\\
   &  V(t,x,y) = M,  & \text{in} \  (T-r,T)\times B^c_{r}(\bar{x}) \times \mathbb{R}^m, \\
   &  V(T,x,y) = g(x,y),  & \text{in} \   B_{r}(\bar{x})\times\mathbb{R}^m.
   \end{aligned}
   \right.
\end{equation}
Consider the upper semicontinuous function $U^{\varepsilon}: [T-r,T] \times \mathbb{R}^n\times \mathbb{R}^m \mapsto \mathbb{R}$ so that $U^{\varepsilon} = u^{\varepsilon}$ in $[T-r,T] \times\bar{B}_{r}(\bar{x})\times \mathbb{R}^m$, and $U^{\varepsilon}  \equiv M$ in $[T-r,T] \times\bar{B}^c_{r}(\bar{x})\times \mathbb{R}^m$. Then $U^{\varepsilon}$ is a subsolution to \eqref{eq:TV-HJB}.
Using strong maximum principle to \eqref{eq:TV-HJB}, for every $\varepsilon \in (0,1)$, $(t,x,y) \in [T-r,T] \times\Bar{B}_{r}(\Bar{x})\times \mathbb{R}^m$, we get
\begin{align}\label{Efin}
    u^{\varepsilon}(t,x,y) - \psi^{\varepsilon}(t,x,y) =  U^{\varepsilon}(t,x,y) - \psi^{\varepsilon}(t,x,y) \leq M - \psi^{\varepsilon}(t,x,y) \leq 0.
\end{align}
For every $(t,x,y) \in (T-r,T)\times\Bar{B}_{r}(\Bar{x})\times \mathbb{R}^m$, we consider the convergence sequence $$(t_{\varepsilon},x_{\varepsilon},y_{\varepsilon}) \subset (T-r,T)\times\Bar{B}_{r}(\Bar{x})\times \mathbb{R}^m,$$ so that $(\varepsilon,t_{\varepsilon},x_{\varepsilon},y_{\varepsilon}) \to (0,t,x,y)$ as $\varepsilon \to 0$.
By \eqref{E:w-g}, we can take the upper limit of both sides of \eqref{Efin} as $(\varepsilon,t',x',y') \to (0,t,x,y)$ and obtain
\begin{equation}
    \Bar{u}(t,x) \leq \bar{g}(x) + \gamma + \psi(x) + C_{r}(T-t).
\end{equation}
for every $\gamma >0$, $t\in (T-r,T)$, $x\in B_{r}(\Bar{x})$.
Letting $(t,x) \to (T,\Bar{x})$, since $\gamma >0 $ is arbitrary and $\psi(\bar{x}) = 0$, we have
\begin{equation}
    \Bar{u}(T,\Bar{x}) \leq \Bar{g}(\Bar{x}).
\end{equation}
The proof for $ \underline{u}(T,\Bar{x}) \geq \Bar{g}(\Bar{x})$ is an analogous argument.

\textbf{Step 5} (locally uniformly convergence).
Since $\Bar{u}(T,\Bar{x}) \leq \Bar{g}(\Bar{x}) \leq \underline{u}(T,\Bar{x})$, using comparison principle we have $\Bar{u} \leq \underline{u}$ in $[0,T]\times \mathbb{R}^m$. However, the definition of relaxed semilimits implies that $\Bar{u} \geq \underline{u}$ in $[0,T]\times \mathbb{R}^m$. Thus $\Bar{u} = \underline{u}:= u$ in $[0,T]\times \mathbb{R}^m$. Moreover, by continuity of $u$ and the definition of relaxed semilimits, $u^{\varepsilon}$ converges locally uniformly to $u$ (see e.g. \cite[Lemma 5.1.9]{MC97}).

\end{proof}

\section{Probabilistic method }

In this section, we establish the averaging principle stated in Theorem~\ref{intm} via a probabilistic approach under Assumptions $(\bf{A^{'}_{b}})$, $(\bf{A_{c}})$, $(\bf{A^{'}_{L}})$, $(\bf{A^{'}_{g}})$ and ($\bf{A_{b,L}}$). More precisely, we demonstrate that the value function of the original multiscale stochastic control system converges to that of the effective reduced system.

\subsection{ Averaging principle of the
controlled jump diffusion }

To prove the averaging principle, we give the following exponential ergodicity for equation \eqref{0nsde}, inspired by \cite[Proposition 1]{YZ24}. The detailed proof is omitted.

\begin{prop}
\label{ergodic11}
Under Hypothesis ($\bf{A_{c}}$), for each function $\widetilde{\varphi}(y) \in C^{1}_{b}$, there exist positive constants $C$ and $\gamma$ such that for all $s \geq 0$ and $x^{\prime} \in \mathbb{R}^n$, we have
\begin{equation*}
\sup_{x^{\prime} \in \mathbb{R}^{n}}|P^{x^{\prime}}_s\widetilde{\varphi}(y)-\mu^{x^{\prime}}(\widetilde{\varphi})|\leq C\|\widetilde{\varphi}\|_{1}e^{-\frac{\beta s}{8}}(1+|y|),
\end{equation*}
where $P^{x^{\prime}}_s\widetilde{\varphi}(y)=\mathbb{E}\widetilde{\varphi}(Y^{x^{\prime},y}_{s})$, and $\mu^{x'}(\widetilde{\varphi}) = \int_{\mathbb{R}^m} \widetilde{\varphi}(y) \mu^{x'}(dy)$.
\end{prop}

In the following, we show that as the scale parameter $\varepsilon \to 0$, the slow component $X^{\varepsilon}$ of the original system strongly converges to the averaged system $\Bar{X}$ in $L^{p}$ sense.  To prove Theorem \ref{intm}, we need to deduce the following moment properties for fast component $Y^{\varepsilon}_{s}$ and slow component $X^{\varepsilon}_{s}$.
\begin{lem}
\label{corollary1}
Suppose ($\bf{A_{b}}$)  hold. Then for every $1 \leq p < \alpha_2$, it holds that
\begin{equation*}
\mathbb{E}\left(\sup_{t \leq s \leq T}\left|Y^{\varepsilon}_{s}\right|^p\right)\leq C(p) \left( T^{\frac{p}{\alpha_2}}\vee T^{1-\frac{1}{\alpha_2}+\frac{p}{\alpha_2}}\right) \varepsilon^{-\frac{p}{\alpha_2}} +|y|^p.
\end{equation*}
This implies that
\begin{equation*}
\varepsilon^{p} \mathbb{E}\left(\sup_{t\leq s \leq T}\left|Y^{\varepsilon}_{s}\right|^p\right)\rightarrow 0, \quad \text{as } \varepsilon \rightarrow 0.
\end{equation*}
\end{lem}
\begin{proof}
Note that for each $\varepsilon >0$, we have
\begin{equation*}
\begin{aligned}
Y^{\varepsilon}_{s\varepsilon} &=y+\frac{1}{\varepsilon}\int_{t}^{s\varepsilon}c({X^{\varepsilon}_r}, Y^{\varepsilon}_r)dr + \frac{1}{\varepsilon^{\frac{1}{\alpha_2}}}L_{s\varepsilon}^{{\alpha_2}}
=y+\int^{s}_{\frac{t}{\varepsilon}}c({X^{\varepsilon}_{r\varepsilon}}, Y^{\varepsilon}_{r\varepsilon})ds+\widetilde{L}^{\alpha_2}_{s},
\end{aligned}
\end{equation*}
where $\{\widetilde{L}^{\alpha_2}_{s}=\varepsilon^{-\frac{1}{\alpha_2}}L_{s\varepsilon}^{{\alpha_2}}, s \geq 0\}$ is an $\alpha$-stable process with the same law as $L^{\alpha_2}_s$.
Using the condition $\sup_{y \in \mathbb{R}^m} |c(0,y)|< \infty$ and the same argument as Lemma 6 in \cite{YZ24}, we can obtain that
\begin{equation*}
\mathbb{E}\left(\sup_{t \leq s \leq T} |Y^{\varepsilon}_{s\varepsilon} |^{p}\right) \leq C(p)\left(T^{\frac{p}{\alpha_2}}\vee T^{1-\frac{1}{\alpha_2}+\frac{p}{\alpha_2}}\right) +|y|^p.
\end{equation*}
Therefore,
\begin{equation*}
  \begin{split}
    \mathbb{E}\left(\sup_{s \in[t,T]}|Y^{\varepsilon}_s|^p\right) &= \mathbb{E}\left(\sup_{\frac{t}{\varepsilon} \leq s \leq \frac{T}{\varepsilon}}|Y^{\varepsilon}_{s\varepsilon} |^{p}\right)\leq C(p) \left( \left(\frac{T}{\varepsilon}\right)^{\frac{p}{\alpha_2}}\vee \left(\frac{T}{\varepsilon}\right)^{1-\frac{1}{\alpha_2}+\frac{p}{\alpha_2}}\right) +|y|^p \\
    &\le C(p) \left( T^{\frac{p}{\alpha_2}}\vee T^{1-\frac{1}{\alpha_2}+\frac{p}{\alpha_2}}\right) \varepsilon^{-\frac{p}{\alpha_2}} +|y|^p.
  \end{split}
\end{equation*}
\end{proof}

By employing a technique similar to that in \cite[Lemma 4]{YZ24}, we obtain


\begin{lem}
\label{lemma-3}
Under Hypothesis ($\bf{A_{c}}$),  for all $t\geq 0$, and $x_i \in \mathbb{R}^n$, $y_i \in \mathbb{R}^m$, $i=1,2$, we have
\begin{equation}
\begin{aligned}
|\nabla_{x}Y^{x_1,y_1}_s-\nabla_{x}Y^{x_2,y_2}_s|^2  &\leq C(\|c\|_2,\beta) s e^{-\frac{\beta}{2}s}|y_1-y_2|^2+C(\|c\|_2,\beta)|x_1-x_2|^2, \\
|\nabla_{y}Y^{x_1,y_1}_s-\nabla_{y}Y^{x_2,y_2}_s|^2 &\leq C(\|c\|_2,\beta) se^{-\frac{\beta s}{2}} \left(|y_1-y_2|^2+|x_1-x_2|^2\right),
\end{aligned}
\end{equation}
where $C(\|c\|_2,\beta)$ is a constant independent of $s$.
\end{lem}

Next, we will use the technique of nonlocal Poisson equation to prove Theorem \ref{intm}.
\begin{lem}
\label{lem3.2}
Suppose that the assumptions in Theorem \ref{intm} hold. Define
\begin{equation}
\label{Phi}
\Phi(x,y,v):=\int^{\infty}_{0}\left[\mathbb{E}b(x,Y^{x,y}_s,v)-\bar{b}(x,v)\right]ds.
\end{equation}
and
\begin{equation}
\label{poisson}
\mathcal{L}_2(x,y)\Phi(x,y,v)=-b(x,y,v)+\bar{b}(x,v).
\end{equation}

Then $\Phi(x, y, v)$ is a solution of the non-autonomous Poisson equation \eqref{poisson}. Moreover, we have  \\
(i)
\begin{equation}
\label{i}
\begin{aligned}
\sup_{v\in U}\sup_{x\in \mathbb{R}^n}\left|\Phi(x,y,v)\right| &\leq C(\|\nabla_y b\|_{0}, \beta)(1+|y|); \\
\end{aligned}
\end{equation}
(ii)
\begin{equation}
\begin{aligned}
\sup_{v\in U}\sup_{y\in \mathbb{R}^n}\left\|\nabla_y\Phi(x,y,v)\right\| &\leq C(\|\nabla_y b\|_{0}, \beta),~~ \sup_{v\in U}\sup_{y\in \mathbb{R}^n}\left\|\nabla^{2}_y\Phi(x,y,v)\right\| & \leq C(\mathcal{P}); \\
\end{aligned}
\end{equation}
(iii)
\begin{equation}
\label{iii}
\begin{aligned}
\sup_{v\in U}\sup_{x\in \mathbb{R}^n}\left\|\nabla_x\Phi(x,y,v)\right\| &\leq C(\mathcal{P})(1+|y|^{\frac{1}{2}});
\end{aligned}
\end{equation}
(iv)
\begin{equation}
\label{iv}
\begin{aligned}
\sup_{v\in U}\sup_{x\in \mathbb{R}^n}\left\|\nabla^2_x\Phi(x,y,v)\right\| &\leq C(\mathcal{P})(1+|y|);
\end{aligned}
\end{equation}
(v)
\begin{equation}
\label{gradient}
\begin{aligned}
\sup_{v\in U}\left\|\nabla_v\Phi(x,y,v)\right\| &\leq C(\|\nabla_y \nabla_v b\|_{0}, \beta).
\end{aligned}
\end{equation}
\end{lem}

\begin{proof}
Set $\Re(x,Y^{x,y}_s,v):=b(x,Y^{x,y}_s,v)-\bar{b}(x,v)$. Then $\Re$ satisfies the following centering condition, i.e.,
\begin{equation*}
\label{equa}
\int_{\mathbb{R}^m}\Re(x,y,v)\mu_{x}(dy)=0.
\end{equation*}
It is well known that $P^{x}_{s}b(x, \cdot, v)(y):=\mathbb{E}b(x,Y^{x,y}_s,v)$ satisfies the backward Fokker-Planck equation
\begin{equation}
\frac{d}{ds}P^{x}_{s}b(x, \cdot, v)(y)=\mathcal{L}_2(x,y)\left[P^{x}_{s}b(x, \cdot, v)(y)\right].
\end{equation}

By combining the definition of $\mathcal{L}_2$ with equation \eqref{equa}, we obtain
\begin{equation}
\begin{aligned}
\mathcal{L}_2(x,y)\Phi(x, y, v)&=\mathcal{L}_2(x,y)\left[\int^{\infty}_{0}\mathbb{E}\Re(x, \cdot, v)(y)ds\right] \\
&=\int^{\infty}_{0}\mathcal{L}_2(x,y)\left[P^{x}_{s}b(x, \cdot, v)(y)\right]ds \\
&=\int^{\infty}_{0}\frac{d}{ds}P^{x}_{s}b(x, \cdot, v)(y)ds\\
&=-b(x,y,v)+\bar{b}(x,v).
\end{aligned}
\end{equation}
Below, we mainly focus on the estimates \eqref{i}-\eqref{gradient}.

(i) By Proposition \ref{ergodic11}, we have
\begin{equation}
\begin{aligned}
\sup_{v\in U}\sup_{x\in \mathbb{R}^n}\left|\Phi(x,y,v)\right| &\leq  \sup_{v\in U}\sup_{x\in \mathbb{R}^n}\int^{\infty}_0\left|\mathbb{E}b(x,Y^{x,y}_s,v)-\bar{b}(x,v)\right|ds\\
&\leq \sup_{v\in U}\sup_{x\in \mathbb{R}^n} \|\nabla_y b\|_{0}\cdot \int^{\infty}_0e^{-\frac{\beta r}{8}}(1+|y|)dr, \\
& \leq C(\|\nabla_y b\|_{0}, \beta)(1+|y|).
\end{aligned}
\end{equation}
(ii) By the equality \eqref{Phi} and inequality \eqref{dY}, we obtain
\begin{equation}
\begin{aligned}
\sup_{v\in U}\sup_{x\in \mathbb{R}^n}\left\|\nabla_y\Phi(x,y,v)\right\| &= \sup_{v\in U}\sup_{x\in \mathbb{R}^n}\int^{\infty}_0 \left|\mathbb{E}\left[\nabla_y b(x,Y^{x,y}_s,v)\cdot \nabla_yY^{x,y}_s\right]\right|ds \\
& \leq \sup_{v\in U}\sup_{x\in \mathbb{R}^n}\int^{\infty}_0 \|\nabla_y b\|_{0}e^{-\frac{\beta s}{2}} ds \\
& \leq C(\|\nabla_y b\|_{0}, \beta).
\end{aligned}
\end{equation}

Similarly, using equality \eqref{Phi} together with inequality \eqref{dY}, we derive
\begin{equation}
\begin{aligned}
\sup_{v\in U}\sup_{x\in \mathbb{R}^n}\left\|\nabla^{2}_y\Phi(x,y,v)\right\|
&= \sup_{v\in U}\sup_{x\in \mathbb{R}^n}\int^{\infty}_0 \left|\mathbb{E}\left[\nabla^{2}_y b(x,Y^{x,y}_s,v)\cdot \nabla_yY^{x,y}_s\right]\right|ds \\
&+\sup_{v\in U}\sup_{x\in \mathbb{R}^n}\int^{\infty}_0 \left|\mathbb{E}\left[\nabla_y b(x,Y^{x,y}_s,v)\cdot \nabla^{2}_yY^{x,y}_s\right]\right|ds        \\
& \leq \sup_{v\in U}\sup_{x\in \mathbb{R}^n}\int^{\infty}_0 \|\nabla^2_y b\|_{0}e^{-\frac{\beta s}{2}} ds
+\sup_{v\in U}\sup_{x\in \mathbb{R}^n}C(\|c\|_2,\beta,\|\nabla_y b\|_{0})\int^{\infty}_0 \sqrt{s} e^{-\frac{\beta s}{4}} ds \\
& \leq C(\mathcal{P}).
\end{aligned}
\end{equation}

(iii) Let us define
\begin{equation*}
\widetilde{b}(s,s_0,x,y,v):=\mathbb{E}b(x,Y^{x,y}_s,v)-\mathbb{E}b(x,Y^{x,y}_{s+s_0},v)
\end{equation*}
and
\begin{equation*}
\hat{b}(s, x,y,v):=\mathbb{E}b(x,Y^{x,y}_s,v).
\end{equation*}

Then, by Proposition \ref{ergodic11} and the Markov property, it follows that
\begin{equation}
\begin{aligned}
\widetilde{{b}}(s,\infty,x,y,v)&=\hat{b}(s, x,y,v)-\bar{b}(s, x,v), \\
\end{aligned}
\end{equation}
and
\begin{equation}
\widetilde{b}(s, s_0,x,y,v)=\hat{b}(s, x,y,v)-\mathbb{E}\hat{b}(s, x,Y^{x,y}_{s_0},v).
\end{equation}
Thus, we have
\begin{equation}
\label{nablaxx}
\begin{aligned}
\nabla_x \widetilde{b}(s, s_0,x,y,v)&=\nabla_x\hat{b}(s,x,y,v)-\mathbb{E}\left[\nabla_x\hat{b}(s,x,Y^{x,y}_{s_0},v)\right]
-\mathbb{E}\left[\nabla_y\hat{b}(s, x,Y^{x,y}_{s_0},v)\cdot \nabla_x Y^{x,y}_{s_0}\right].
\end{aligned}
\end{equation}

By Proposition \ref{ESYxy} and Lemma \ref{lemma-3}, we have
\begin{equation}
\label{nablax}
\begin{aligned}
\sup_{x\in \mathbb{R}^n, y\in \mathbb{R}^{m}}\|\nabla_x Y^{x,y}_s\| &\leq C(\|c\|_1,\beta), ~~~~~~~
\sup_{x\in \mathbb{R}^n, y\in \mathbb{R}^{m}}\|\nabla_y Y^{x,y}_s\| \leq e^{-\frac{\beta}{4}s},\\
\sup_{x\in \mathbb{R}^n, y\in \mathbb{R}^{m}}\|\nabla_x\nabla_y Y^{x,y}_s\| &\leq C(\|c\|_2,\beta) \sqrt{s}e^{-\frac{\beta s}{4}},~~ \sup_{x\in \mathbb{R}^n, y\in \mathbb{R}^{m}}\|\nabla^2_x Y^{x,y}_s\| \leq C(\|c\|_2,\beta), \\
\sup_{x\in \mathbb{R}^n, y\in \mathbb{R}^{m}}\|\nabla^2_y Y^{x,y}_s\| &\leq C(\|c\|_2,\beta) \sqrt{s}e^{-\frac{\beta s}{4}}.
\end{aligned}
\end{equation}

Note that
\begin{equation}
\label{results}
\begin{aligned}
\nabla_x\hat{b}(s,x,y,v)&=\mathbb{E}\left[\nabla_xb(x,Y^{x,y}_s,v)+\nabla_yb(x,Y^{x,y}_s,v)\cdot \nabla_x Y^{x,y}_s\right],\\
\nabla_y\hat{b}(s,x,y,v)&=\mathbb{E}\left[\nabla_yb\left(x,Y^{x,y}_s,v\right)\cdot \nabla_y Y^{x,y}_s\right],\\
\nabla^2_x\hat{b}(s,x,y,v)&=\mathbb{E}\left[\nabla^2_xb(x,Y^{x,y}_s,v)\right]
+2\mathbb{E}\left[\nabla_x\nabla_yb(x,Y^{x,y}_s,v)\cdot \nabla_x Y^{x,y}_s\right]\\&+\mathbb{E}\left[\nabla^2_yb(x,Y^{x,y}_s,v)\cdot \left(\nabla_x Y^{x,y}_s\right)^2\right]+\mathbb{E}\left[\nabla_yb(x,Y^{x,y}_s,v)\cdot \nabla^2_x Y^{x,y}_s\right],\\
\nabla_y\nabla_x\hat{b}(s,x,y,v)&=\mathbb{E}\left[\nabla_y\nabla_xb(x,Y^{x,y}_s,v)\cdot \nabla_y Y^{x,y}_s+\nabla^2_yb(x,Y^{x,y}_s,v)\cdot \nabla_y Y^{x,y}_s\cdot \nabla_x Y^{x,y}_s\right],\\
&+\mathbb{E}\left[\nabla_yb(x,Y^{x,y}_s,v)\cdot \nabla_y \nabla_x Y^{x,y}_s\right], \\
\nabla^2_y\hat{b}(s,x,y,v)&=\mathbb{E}\left[\nabla^{2}_yb\left(x,Y^{x,y}_s,v\right)\cdot \nabla_y Y^{x,y}_s\right]
+\mathbb{E}\left[\nabla_yb\left(x,Y^{x,y}_s,v\right)\cdot \nabla^{2}_y Y^{x,y}_s\right].
\end{aligned}
\end{equation}
Therefore by the assumption $b$, inequality \eqref{nablax} and equality \eqref{results}, we have
\begin{equation}
\label{result}
\begin{aligned}
\|\nabla_y\hat{b}(s,x,y,v)\|_{0}&\leq \|\nabla_y{b}(x,y,v)\|_{0}\cdot e^{-\frac{\beta}{4}s},\\
\|\nabla_x\nabla_y\hat{b}(s,x,y,v)\|_{0}& \leq \|\nabla_y\nabla_x{b}(x,y,v)\|_{0}\cdot e^{-\frac{\beta}{4}s}+
\|\nabla^2_y{b}(x,y,v)\|_{0}\cdot C(\|c\|_1,\beta)e^{-\frac{\beta s}{4}}\\
&+\|\nabla_y{b}(x,y,v)\|_{0}\cdot C(\|c\|_2,\beta)\sqrt{s}e^{-\frac{\beta s}{4}},\\
\|\nabla^2_y\hat{b}(s,x,y,v)\|_{0}& \leq \|\nabla^2_y{b}(x,y,v)\|_{0}\cdot e^{-\frac{\beta s}{4}}+ \|\nabla_y{b}(x,y,v)\|_{0}\cdot C(\|c\|_2,\beta)\sqrt{s}e^{-\frac{\beta s}{4}},
\end{aligned}
\end{equation}
and
\begin{equation}
\label{0nabla}
\begin{aligned}
\|\nabla_x \hat{b}(s,x,y_1,v)-\nabla_x \hat{b}(s,x,y_2,v)\|&\leq \left\|\mathbb{E}\left[\nabla_xb(x,Y^{x,y_1}_s,v)
-\nabla_xb(x,Y^{x,y_2}_s,v)\right]\right\|\\
&+ \left\|\mathbb{E}\left[\left(\nabla_yb(x,Y^{x,y_1}_s,v)
-\nabla_yb(x,Y^{x,y_2}_s,v)\right)\cdot \nabla_x Y^{x,y_1}_s\right]\right\|\\
&+\left\|\mathbb{E}\left[\nabla_yb(x,Y^{x,y_2}_s,v)\cdot
\left(\nabla_x Y^{x,y_1}_s-\nabla_x Y^{x,y_2}_s\right)\right]\right\|\\
& \leq C(\|\nabla_x b(x,y,v)\|_0, \|\nabla_y\nabla_x b(x,y,v)\|_0) \mathbb{E}\left[|Y^{x,y_1}_s-Y^{x,y_2}_s|^{\frac{1}{2}}\right]\\
&+C(\|c\|_1,\beta,\|\nabla^2_y b(x,y,v)\|_0) \mathbb{E}\left[|Y^{x,y_1}_s-Y^{x,y_2}_s|^{\frac{1}{2}}\right]\\
&+ \|\nabla_y b(x,y,v)\|_0 \cdot C(\|c\|_2,\beta) \sqrt{s}e^{-\frac{\beta s}{4}}\\
& \leq C(\|\nabla_x b(x,y,v)\|_0, \|\nabla_y\nabla_x b(x,y,v)\|_0) e^{-\frac{\beta}{4}s}|y_1-y_2|^{\frac{1}{2}}\\
&+C(\|c\|_1,\beta,\|\nabla^2_y b(x,y,v)\|_0  ) \cdot  e^{-\frac{\beta}{4}s}|y_1-y_2|^{\frac{1}{2}} \\
&+C(\|c\|_2,\beta, \|\nabla_y b(x,y,v)\|_0) \sqrt{s}e^{-\frac{\beta s}{4}}|y_1-y_2|^{\frac{1}{2}}.
\end{aligned}
\end{equation}

Taking the inequalities \eqref{0nabla} and \eqref{nablax} in \eqref{nablaxx},  we have
\begin{equation}
\begin{aligned}
&\sup_{v\in U}\sup_{x\in \mathbb{R}^n}\left\|\nabla_x\Phi(x,y,v)\right\|= \sup_{v\in U}\sup_{x\in \mathbb{R}^n}\left\|\int^{\infty}_{0}\nabla_x\widetilde{{b}}(s,\infty,x,y,v)ds\right\|\\
&=\sup_{v\in U}\sup_{x\in \mathbb{R}^n}\left\|\int^{\infty}_{0}\lim_{s_0\rightarrow \infty}\nabla_x\widetilde{{b}}(s,s_0,x,y,v)ds\right\|\\
&=\sup_{v\in U}\sup_{x\in \mathbb{R}^n}\left\|\int^{\infty}_{0}\lim_{s_0\rightarrow \infty}\left\{\nabla_x\hat{b}(s,x,y,v)-\mathbb{E}\left[\nabla_x\hat {b}(s, x,Y^{x,y}_{s_0},v)\right]-\mathbb{E}\left[\nabla_y\hat{b}(s,x,Y^{x,y}_{s_0},v)\cdot \nabla_x Y^{x,y}_{s_0}\right]\right\} ds\right\|\\
& \leq C(\mathcal{P})(1+|y|^{\frac{1}{2}}).
\end{aligned}
\end{equation}
(iv) Recall that
\begin{equation}
\nabla_x\widetilde{b}(s,s_0,x,y,v)=\nabla_x\hat{b}(s,x,y,v)-\mathbb{E}\nabla_x\hat{b}(s,x,Y^{x,y}_{s_0},v)
-\mathbb{E}\left[\nabla_y\hat{b}(s,x,Y^{x,y}_{s_0},v)\cdot \nabla_x Y^{x,y}_{s_0}\right],
\end{equation}
and
\begin{equation}
\begin{aligned}
&\nabla_x\widetilde{b}(s,s_0,x_1,y,v)-\nabla_x\widetilde{b}(s,s_0,x_2,y,v)
=\nabla_x\hat{b}(s,x_1,y,v)-\mathbb{E}\nabla_x\hat{b}(s,x_1,Y^{x_1,y}_{s_0},v_{s^{'}})-\mathbb{E}\left[\nabla_y\hat{b}(s,x_1,Y^{x_1,y}_{s_0},v)\cdot \nabla_x Y^{x_1,y}_{s_0}\right]\\ &-\left\{\nabla_x\hat{b}(s,x_2,y,v)-\mathbb{E}\nabla_x\hat{b}(s,x_2,Y^{x_2,y}_{s_0},v)-\mathbb{E}\left[\nabla_y\hat{b}(s,x_2,Y^{x_2,y}_{s_0},v)\cdot \nabla_x Y^{x_2,y}_{s_0}\right]\right\}\\
&=\left\{\nabla_x\hat{b}(s,x_1,y,v)-\mathbb{E}\nabla_x\hat{b}(s,x_1,Y^{x_1,y}_{s_0},v)
-\left[\nabla_x\hat{b}(s,x_2,y,v)-\mathbb{E}\nabla_x\hat{b}(s,x_2,Y^{x_1,y}_{s_0},v)\right]\right\}\\
&+\left[\mathbb{E}\nabla_x\hat{b}(s,x_2,Y^{x_2,y}_{s_0},v)-\mathbb{E}\nabla_x\hat{b}(s,x_2,Y^{x_1,y}_{s_0},v)\right]\\
&+\left\{\mathbb{E}\left[\nabla_y\hat{b}(s,x_1,Y^{x_1,y}_{s_0},v)\cdot \nabla_x Y^{x_1,y}_{s_0}]-\mathbb{E}[\nabla_y\hat{b}(s,x_2,Y^{x_2,y}_{s_0},v)\cdot \nabla_x Y^{x_2,y}_{s_0}\right]\right\}\\
&=:\Upsilon_1+\Upsilon_2+\Upsilon_3.
\end{aligned}
\end{equation}
For the term $\Upsilon_1$, note that
\begin{equation}
\nabla_x\hat{b}(s,x,y,v)=\mathbb{E}\left[\nabla_xb(x,Y^{x,y}_s,v)+\nabla_yb(x,Y^{x,y}_s,v)\cdot \nabla_x Y^{x,y}_s\right].
\end{equation}
Thus we have
\begin{equation}
\begin{aligned}
&\left\|\nabla_x\hat{b}(s,x_1,y_1,v)-\nabla_x\hat{b}(s,x_1,y_2,v)-\left[\nabla_x\hat{b}(s, x_2,y_1,v)-\nabla_x\hat{b}(s, x_2,y_2,v)\right]\right\|
\\
&\leq \left\|\mathbb{E}\left[\nabla_x{b}(x_1,Y^{x_1,y_1}_s,v)-\nabla_x{b}(x_1,Y^{x_1,y_2}_s,v)-\left(\nabla_x{b}(x_2,Y^{x_1,y_1}_s,v)
-\nabla_x{b}(x_2,Y^{x_1,y_2}_s,v)\right)\right]\right\|\\
&+\left\|\mathbb{E}\left[\nabla_x{b}(x_2,Y^{x_1,y_1}_s,v)-\nabla_x{b}(x_2,Y^{x_2,y_1}_s,v)-\left(\nabla_x{b}(x_2,Y^{x_1,y_2}_s,v)
-\nabla_x{b}(x_2,Y^{x_2,y_2}_s,v)\right)\right]\right\|\\
&+\left\|\mathbb{E}\left[\nabla_y{b}(x_1,Y^{x_1,y_1}_s,v)\cdot\nabla_x Y^{x_1,y_1}_s-\nabla_y{b}(x_2,Y^{x_2,y_1}_s,v)\cdot\nabla_x Y^{x_2,y_1}_s
\right.\right.\\&\left.\left.-\mathbb{E}[\nabla_y{b}(x_1,Y^{x_1,y_2}_s,v)]\cdot\nabla_x Y^{x_1,y_2}_s-\nabla_y{b}(x_2,Y^{x_2,y_2}_s,v)]\cdot\nabla_x Y^{x_2,y_2}_s\right]\right\|.
\end{aligned}
\end{equation}
By the assumption $b$, inequalities \eqref{dY} and \eqref{nablax}, we have
\begin{equation}
\begin{aligned}
&\left\|\nabla_x\hat{b}(s, x_1,y_1,v)-\nabla_x\hat{b}(s, x_1,y_2,v)-\left[\nabla_x\hat{b}(s, x_1,y_1,v)-\nabla_x\hat{b}(s, x_1,y_1,v)\right]\right\|\\
&\leq C(\mathcal{P})\left(\sqrt{s_0}+1\right)e^{-\frac{\beta s_0}{4}}|x_1-x_2||y_1-y_2|
\end{aligned}
\end{equation}
This implies
\begin{equation*}
\begin{aligned}
|\Upsilon_1| &\leq C(\mathcal{P})\cdot
(x_1-x_2)\cdot\left(\sqrt{s_0}+1\right)e^{-\frac{\beta s_0}{4}}\cdot \mathbb{E}|y-Y^{x_1,y}_{s_0}|.
\end{aligned}
\end{equation*}
Applying Lemma \ref{corollary1} with $\varepsilon = 1$, we obtain
\begin{equation}
\label{Upsilon1}
|\Upsilon_1| \leq C(\mathcal{P})\left(\sqrt{s_0}+1\right)e^{-\frac{\beta s_0}{4}}\cdot |x_1-x_2|\cdot (1+|y|).
\end{equation}
For the term $\Upsilon_2$, using Assumption ($\bf{A_{b}}$), equality \eqref{results}, inequalities \eqref{dY} and \eqref{nablax}, we have
\begin{equation}
\begin{aligned}
|\Upsilon_2| & \leq \|\nabla_y\nabla_x\hat{b}(x,y,v)\|_0 \cdot \mathbb{E}\left|Y^{x_1,y}_{s_0}-Y^{x_2,y}_{s_0}\right|\\
& \leq \|\nabla_y\nabla_x\hat{b}(x,y,v)\|_0 \cdot C(\|c\|_1,\beta)|x_1-x_2|\\
& \leq C(\mathcal{P})\left(\sqrt{s_0}+1\right)e^{-\frac{\beta s_0}{4}}|x_1-x_2|.
\end{aligned}
\end{equation}
For the term $\Upsilon_3$, by a similar argument as $\Upsilon_2$, we have
\begin{equation}
\label{Upsilon3}
\begin{aligned}
|\Upsilon_3| & \leq  \mathbb{E}\left[\|\nabla^{2}_y\hat{b}(x,y,v)\|_0 \cdot\left|Y^{x_1,y}_{s_0}-Y^{x_2,y}_{s_0}\right|\cdot |\nabla_{x}Y^{x_1,y}_{s_0}|\right]\\
&+\mathbb{E}\left[\|\nabla_x\nabla_y\hat{b}(x,y,v)\|_0 \cdot\left|x_1-x_2\right|\cdot |\nabla_{x}Y^{x_1,y}_{s_0}|\right]\\
&+\mathbb{E}\left[\|\nabla_y\hat{b}(x,y,v)\|_0 \cdot |\nabla_{x}Y^{x_1,y}_{s_0}-\nabla_{x}Y^{x_2,y}_{s_0}|\right]\\
&\leq C(\mathcal{P})\left(\sqrt{s_0}+1\right)e^{-\frac{\beta s_0}{4}}|x_1-x_2|.
\end{aligned}
\end{equation}
Combing \eqref{Upsilon1}-\eqref{Upsilon3}, we obtain
\begin{equation}
\begin{aligned}
\left|\nabla_x\widetilde{b}(s,s_0,x_1,y,v)-\nabla_x\widetilde{b}(s,s_0,x_2,y,v)\right|\leq C(\mathcal{P})\left(\sqrt{s_0}+1\right)e^{-\frac{\beta s_0}{4}}\cdot |x_1-x_2|\cdot (1+|y|)
\end{aligned}
\end{equation}
Thus
\begin{equation}
\begin{aligned}
\sup_{v\in U}\sup_{x\in \mathbb{R}^n}\left\|\nabla^2_x\Phi(x,y,v)\right\|&= \sup_{v\in U}\sup_{x\in \mathbb{R}^n}\left\|\int^{\infty}_{0}\nabla^2_x\widetilde{{b}}(s,\infty,x,y,v)ds\right\|\\
&=\sup_{v\in U}\sup_{x\in \mathbb{R}^n}\left\|\lim_{s_0\rightarrow \infty}\int^{\infty}_{0}\nabla^2_x\widetilde{{b}}(s,s_0,x,y,v)ds\right\|\\
& \leq C(\mathcal{P})(1+|y|).
\end{aligned}
\end{equation}
(v) Recall that
\begin{equation}
\begin{aligned}
\sup_{x\in \mathbb{R}^n}\sup_{y\in \mathbb{R}^m}\left\|\nabla_v\Phi(x,y,v)\right\|
&= \sup_{x\in \mathbb{R}^n}\sup_{y\in \mathbb{R}^m}\int^{\infty}_0 \int_{\mathbb{R}^m}\mathbb{E}\left[ \nabla_v b(x,Y^{x,y}_s,v)-\nabla_v b(x,Y^{x,y}_{\infty},v)\right]\mu_{x}(dy)ds \\
& \leq \sup_{v\in U}\sup_{x\in \mathbb{R}^n}\int^{\infty}_0 \|\nabla_y \nabla_v b\|_{0}e^{-\frac{\beta s}{2}} ds \\
& \leq C(\|\nabla_y \nabla_v b\|_{0}, \beta).
\end{aligned}
\end{equation}
\end{proof}
\subsection{Proof of Theorem \ref{intm}}
Now, we are in the position to give proof of Theorem \ref{intm} under assumptions $(\bf{A^{'}_{b}})$, $(\bf{A_{c}})$, $(\bf{A^{'}_{L}})$, $(\bf{A^{'}_{g}})$ and ($\bf{A_{b,L}}$), i.e., we show that the value function $u^{\varepsilon}$ converges to the effective value function $u$. We divide the proof into the following two steps. \\
{ \bf{Step 1.}} Note that the difference between $X^{\varepsilon}_s$ and $\bar{X}^{\varepsilon}_s$ can be expressed as:
\begin{equation}
\begin{aligned}
X^{\varepsilon}_s-\bar{X}^{\varepsilon}_s&=\int^{s}_{t}\left[b(X^{\varepsilon}_r,Y^{\varepsilon}_r, v_r)-\Bar{b}(\Bar{X}^{\varepsilon}_{r},
v_{r})\right]dr \\
&=\int^{s}_{t}\left[b(X^{\varepsilon}_r,Y^{\varepsilon}_r, v_r)-\Bar{b}({X}^{\varepsilon}_{r},
v_{r})\right]dr+\int^{s}_{t}\left[\bar{b}(X^{\varepsilon}_r, v_r)-\Bar{b}(\Bar{X}^{\varepsilon}_{r},
v_{r})\right]dr.
\end{aligned}
\end{equation}
By virtue of the Lipschitz continuity of $\Bar{b}$, we can
\begin{equation}
\begin{aligned}
\sup_{v \in U}\mathbb{E}\left(\sup_{s \in[t,T]}|X^{\varepsilon}_s-\Bar{X}^{\varepsilon}_s|^{p}\right) &\leq C(p)\sup_{v \in U}\mathbb{E}\left[\sup_{s \in[t,T]}\int^{s}_{t}\left|b(X^{\varepsilon}_r,Y^{\varepsilon}_r, v_r)-\Bar{b}({X}^{\varepsilon}_{r},
v_{r})\right|dr\right]\\&+C(p,T)\sup_{v \in U}\mathbb{E}\left(\sup_{s \in[t,T]}|X^{\varepsilon}_s-\Bar{X}^{\varepsilon}_s|^{p}\right).
\end{aligned}
\end{equation}
Gr\"onwall's inequality implies

\begin{equation}
\label{329}
\sup_{v \in U}\mathbb{E}\left(\sup_{s \in[t,T]}|X^{\varepsilon}_s-\Bar{X}^{\varepsilon}_s|^{p}\right) \leq C(p,T)\sup_{v \in U}\mathbb{E}\left[\sup_{s \in[t,T]}\int^{s}_{t}\left|b(X^{\varepsilon}_r,Y^{\varepsilon}_r, v_r)-\Bar{b}({X}^{\varepsilon}_{r},
v_{r})\right|dr\right].
\end{equation}

By It\^o's formula for the function $\Phi(t, x,y,v)$, we have
\begin{equation}
\begin{aligned}
\Phi(X^{\varepsilon}_s,Y^{\varepsilon}_s, v_s)&=\Phi(x,y,v_t)+\int^{s}_t \frac{\partial \Phi}{\partial v} \frac{dv}{dr}dr
+\int^{s}_t\mathcal{L}_1(X^{\varepsilon}_{r},Y^{\varepsilon}_{r})\Phi( X^{\varepsilon}_r,Y^{\varepsilon}_r,v_r)dr\\
&+\frac{1}{\varepsilon}\int^{s}_t\mathcal{L}_2(X^{\varepsilon}_{r},Y^{\varepsilon}_{r})\Phi(X^{\varepsilon}_r,Y^{\varepsilon}_r, v_r)dr+M^{\varepsilon,1}_s+M^{\varepsilon,2}_s,
\end{aligned}
\end{equation}
where $M^{\varepsilon,1}_s$, $M^{\varepsilon,2}_s$ are two $\mathcal{F}_s$-martingales defined by
\begin{equation}
\begin{aligned}
M^{\varepsilon,1}_s&:=\int^{s}_t\int_{\mathbb{R}^n}\Phi(X^{\varepsilon}_{r-}+x,Y^{\varepsilon}_{r-},v_{r-})
-\Phi(X^{\varepsilon}_{r-},Y^{\varepsilon}_{r-},v_{r-})\widetilde{N}^{1}(dr,dx), \\
M^{\varepsilon,2}_s&:=\int^{s}_t\int_{\mathbb{R}^m}\Phi(X^{\varepsilon}_{r-},Y^{\varepsilon}_{r-}+y,v_{r-})
-\Phi(X^{\varepsilon}_{r-},Y^{\varepsilon}_{r-},v_{r-})\widetilde{N}^{2}(dr,dy).
\end{aligned}
\end{equation}
Consequently, we have
\begin{equation}
\begin{aligned}
\int^{s}_{t}-\mathcal{L}_2(X^{\varepsilon}_{r},Y^{\varepsilon}_{r})\Phi(X^{\varepsilon}_{r-},Y^{\varepsilon}_{r-},v_{r-})dr&=\varepsilon \left[\Phi( x,y,v_t)-\Phi(X^{\varepsilon}_{s},Y^{\varepsilon}_{s},v_s)\right]\\
&+\varepsilon \left[\int^{s}_t \frac{\partial \Phi}{\partial v} dv_r+\int^{s}_{t}\mathcal{L}_1(X^{\varepsilon}_{r},Y^{\varepsilon}_{r})\Phi(X^{\varepsilon}_{r},
Y^{\varepsilon}_{r},v_r)dr+M^{\varepsilon,1}_s+M^{\varepsilon,2}_s\right],
\end{aligned}
\end{equation}
and
\begin{equation}
\begin{aligned}
\sup_{v \in U}\mathbb{E}\left(\sup_{s \in[t,T]}|X^{\varepsilon}_s-\Bar{X}^{\varepsilon}_s|^{p}\right)&\leq C(p,T)\sup_{v \in U}\mathbb{E}\left[\sup_{s \in[t,T]}\int^{s}_{t}\left[b(X^{\varepsilon}_r,Y^{\varepsilon}_r, v_r)-\Bar{b}({X}^{\varepsilon}_{r},
v_{r})\right]dr\right]\\
&=
C(p,T)\sup_{v \in U} \mathbb{E}\left[\sup_{s\in[t,T]}\left|\int^{s}_{t}-\mathcal{L}_2(X^{\varepsilon}_{r},Y^{\varepsilon}_{r})
\Phi(X^{\varepsilon}_{r},Y^{\varepsilon}_{r},v_r)dr\right|^{p}\right]\\
&\leq C(p,T)\varepsilon^{p}\sup_{v \in U}\left[\mathbb{E}\sup_{s \in[t,T]}|\Phi(x,y,v_t)-\Phi(X^{\varepsilon}_{s},Y^{\varepsilon}_{s},v_s)|^{p}\right]\\
&+C(p,T)\varepsilon^{p} \sup_{v \in U}\mathbb{E}\int^{T}_{t}\left|\frac{\partial \Phi}{\partial v}\frac{\partial v}{\partial r}\right| dr\\
&+C(p,T)\varepsilon^{p} \sup_{v \in U}\mathbb{E}\int^{T}_{t}\left|\mathcal{L}_1(X^{\varepsilon}_{r},Y^{\varepsilon}_{r},v_r)\Phi(X^{\varepsilon}_{r-},Y^{\varepsilon}_{r-},v_r)\right|^pdr\\
&+C(p,T)\varepsilon^{p}\sup_{v \in U}\mathbb{E}\left(\sup_{s\in[t,T]}|M^{\varepsilon,1}_s|^{p}\right)
\\&+C(p,T)\varepsilon^{p}\sup_{v \in U}\mathbb{E}\left(\sup_{s\in[t,T]}|M^{\varepsilon,2}_s|^{p}\right)\\
&:=C(p,T)\varepsilon^{p}\left(H_1+H_2+H_3+H_4+H_5\right).
\end{aligned}
\end{equation}
For the term $H_1$, by the inequality \eqref{i} and Lemma \ref{corollary1}, we have
\begin{equation}
\begin{aligned}
\label{H1}
H_1 &\leq C(\|\nabla_y b\|_{0}, \beta)\left[\mathbb{E}\sup_{s \in[t,T]}\left(1+|y|+|Y^{\varepsilon}_{s}|\right)^{p}\right]\\
&\leq C(\mathcal{P})(1+|y|^{p})\varepsilon^{-p/\alpha_2}.
\end{aligned}
\end{equation}
For the term $H_2$, by the inequality \eqref{gradient}, we have
\begin{equation}
\begin{aligned}
H_2&\leq C_{T} \sup_{v \in U}\mathbb{E}\int^{T}_{t}\left|\frac{\partial \Phi}{\partial v}\right| dr \leq C(\mathcal{P}).
\end{aligned}
\end{equation}
For the term $H_3$, by Lemma \ref{corollary1}, inequalities \eqref{iii} and \eqref{iv},  we have
\begin{equation}
\label{H2}
\begin{aligned}
H_3&=C(p)\sup_{v \in U}\mathbb{E}\left[\int^{T}_{t}\left|\left\{\int_{\mathbb{R}^n}\left[\Phi(X^{\varepsilon}_{r-}+z, Y^{\varepsilon}_{r-}, v_{r-})-\Phi(X^{\varepsilon}_{r-}, Y^{\varepsilon}_{r-}, v_{r-})-\right.\right.\right.\right.\\
&\quad \quad \left.\left.\left.\left.I_{\{|z|\leq 1\}}\left\langle z, \nabla_{x} \Phi(X^{\varepsilon}_{r-}, Y^{\varepsilon}_{r-}, v_{r-})\right\rangle \right]v_{1}(dz)\right\}\right|^{p}dr\right]\\
&+C(p)\sup_{v \in U}\mathbb{E}\left[\int^{T}_{t}\left|\langle b(X^{\varepsilon}_{r},Y^{\varepsilon}_{r}, v^{\varepsilon}_{r}), \nabla_{x} \Phi(X^{\varepsilon}_{r}, Y^{\varepsilon}_{r}, v^{\varepsilon}_{r}) \rangle \right|^p\right]dr\\
& \leq C(p)\sup_{v \in U} \mathbb{E}\int^T_{t}\left(\|\nabla_{x} \Phi \|_{0}\int_{|z|>1}z\nu_1(dz)+\|\nabla^2_{x} \Phi \|_0\int_{|z| \leq 1}z^2\nu_1(dz)\right)^{p}ds\\
&+C(p, K_2)\sup_{v \in U} \mathbb{E}\int^T_{t}(1+|X^{\varepsilon}_{r}|^p)\left(1+|Y^{\varepsilon}_{r}|^{p/2}\right)dr\\
& \leq C(p)\sup_{v \in U} \mathbb{E}\int^T_{t}\left(\|\nabla_{x} \Phi \|_{0}\int_{|z|>1}z\nu_1(dz)+\|\nabla^2_{x} \Phi \|_0\int_{|z| \leq 1}z^2\nu_2(dz)\right)^pds\\
&+C(p, K_2)\sup_{v \in U} \mathbb{E}\int^T_{t}\left(1+|X^{\varepsilon}_{r}|^{q_1}+|Y^{\varepsilon}_{r}|^{q_2}\right)dr\\
& \leq C(p,T)\left(1+|x|^{q_1}+|y|^{q_2}\right),
\end{aligned}
\end{equation}
where $p<q_1<\alpha_1$ and  $p<q_2<\alpha_2$.

For the term $H_4$, by applying the Burkholder-Davis-Gundy inequality, along with inequality \eqref{iii} and Lemma \ref{corollary1}, we obtain
\begin{equation}
\begin{aligned}
H_4 &\leq C(p)\sup_{v\in U}\mathbb{E}\left[\sup_{s \in[t,T]}\left|\int^{s}_t\int_{|x|\leq 1}\Phi(t,X^{\varepsilon}_{r-}+x,Y^{\varepsilon}_{r-},v_{r-})
-\Phi(t,X^{\varepsilon}_{r-},Y^{\varepsilon}_{r-},v_{r-})\widetilde{N}^{1}(dr,dx)\right|^{p}\right]\\
&+C(p)\sup_{v\in U}\mathbb{E}\left[\sup_{s \in[t,T]}\left|\int^{s}_t\int_{|x|>1}\Phi(t,X^{\varepsilon}_{r-}+x,Y^{\varepsilon}_{r-},v_{r-})
-\Phi(t,X^{\varepsilon}_{r-},Y^{\varepsilon}_{r-},v_{r-})\widetilde{N}^{1}(dr,dx)\right|^{p}\right]\\
& \leq C(p)\sup_{v\in U}\mathbb{E}\left|\int^{T}_t\int_{|x|\leq 1}\left|\Phi(t,X^{\varepsilon}_{r-}+x,Y^{\varepsilon}_{r-},v_{r-})
-\Phi(t,X^{\varepsilon}_{r-},Y^{\varepsilon}_{r-},v_{r-})\right|^{2}{N}^{1}(dr,dx)\right|^{\frac{p}{2}}\\
&+C(p)\sup_{v\in U}\mathbb{E}\left[\int^{T}_t\int_{|x|>1}\left|\Phi(t,X^{\varepsilon}_{r-}+x,Y^{\varepsilon}_{r-},v_{r-})
-\Phi(t,X^{\varepsilon}_{r-},Y^{\varepsilon}_{r-},v_{r-})\right|^{p}\nu_1(dx)ds\right]\\
& \leq  C(p,T) \sup_{v\in U}\mathbb{E}\int^{T}_{t}\left(\left(\int_{|x|\leq 1}|x|^2\nu_1(dx)\right)^{\frac{p}{2}}+\left(\int_{|x|> 1}|x|^p\nu_1(dx)\right)\right)\left(1+|Y^{\varepsilon}_{r}|^{p}\right)dr\\
& \leq C(p,T)(1+|y|^p)\varepsilon^{-p/\alpha_2}.
\end{aligned}
\end{equation}
Using the same argument as $H_4$, we have
\begin{equation}
\label{H34}
H_5 \leq C(p,T)\varepsilon ^{-\frac{p}{\alpha_2}}.
\end{equation}
Combining \eqref{H1}, \eqref{H2} and \eqref{H34}, we obtain
\begin{equation*}
\label{XX}
\sup_{v \in U}\mathbb{E}\left(\sup_{s\in[t,T]}|X^{\varepsilon}_s-\Bar{X}^{\varepsilon}_s|^{p}\right)\leq C(\mathcal{P}) \varepsilon^{p\left(1-{1}/{\alpha_2}\right)}.
\end{equation*}

{ \bf{Step 2.}} Let $v^{\varepsilon}$ be any admissible control, let $(X^{\varepsilon},Y^{\varepsilon})$ be the solution
of \eqref{slo0} corresponding to $v^{\varepsilon}$, and let $\Bar{X}^{\varepsilon}$ be the solution of \eqref{barX} corresponding to the same $v^{\varepsilon}$. By Lipschitz property of function $L$, Lemma \ref{intm} and the same argument as the right side of \eqref{329}, we have as $\varepsilon \rightarrow 0$

\begin{equation}
\label{N}
\begin{aligned}
& \sup_{v^{\varepsilon} \in U}\mathbb{E}\left[ \left|\int^{T}_{t}e^{-\lambda(s-t)}\left[L(X^{\varepsilon}_s,Y^{\varepsilon}_s,v^{\varepsilon}_s)-\bar{L}(\bar{X}^{\varepsilon}_s, v^{\varepsilon}_s)\right]ds \right|\right]\\
&\leq  \sup_{v^{\varepsilon} \in U}\mathbb{E}\left[\int^{T}_{t}\left|L(X^{\varepsilon}_s,Y^{\varepsilon}_s,v^{\varepsilon}_s)-{L}(\bar{X}^{\varepsilon}_s, Y^{\varepsilon}_s,v^{\varepsilon}_s)\right|ds\right]\\&+ \sup_{v^{\varepsilon} \in U}\mathbb{E}\left[\mathbb{E}\int^{T}_{t}\left|{L}(\bar{X}^{\varepsilon}_s, Y^{\varepsilon}_s,v^{\varepsilon}_s)-\bar{L}(\bar{X}^{\varepsilon}_s,v^{\varepsilon}_s)\right|ds\right] \\
& \leq C(\mathcal{P}) \varepsilon^{p\left(1-{1}/{\alpha_2}\right)}.
\end{aligned}
\end{equation}

Introduce the following Kolmogorov equation
\begin{equation}
\label{KE}
\left\{
\begin{aligned}
\partial_t \Upsilon(s,x)&=\bar{\mathcal{L}}_1 \Upsilon(s,x), ~~ s \in [t,T], \\
\Upsilon(t,x)&=g(x),\\
\end{aligned}
\right.
\end{equation}
where $g \in \cC^{3}_{b}$ and $\bar{\mathcal{L}}_1$ is the infinitesimal generator of the transition semigroup of the averaged equation \eqref{barX}, which is given by
\begin{equation}
\bar{\mathcal{L}}_1 \Upsilon(s,x):=-\left(-\Delta_{x}\right)^{\frac{\alpha_1}{2}}\Upsilon(s,x)+\bar{b}(x,v) \cdot \nabla_{x} \Upsilon(s,x).
\end{equation}
One can check by straightforward computation that $\bar{b}\in \cC^{3}_{b}$. Thus, equation \eqref{KE} has a unique solution $\Upsilon(s,x)=\mathbb{E}g(\bar{X}_s(x))$. We define $\Upsilon^{T+t}(s,x):=\Upsilon(T+t-s,x)$. If follows that  $\Upsilon^{T+t}(T,X^{\varepsilon}_T)=\mathbb{E}g(X^{\varepsilon}_T)$ and $\Upsilon^{T+t}(t,x)=\mathbb{E}g(\bar{X}_T(x))$.
By applying It\^o's formula to the terminal cost function $g$ and using an argument analogous to the right-hand side of \eqref{329} for the slow component $X^{\varepsilon}_{t}$ as in \cite[Theorem 2.3]{XS22}, we obtain that
\begin{align}
\label{CovG}
  \sup_{v^{\varepsilon} \in U}  \mathbb{E}\left|e^{\lambda(t-T)}{g}\left(X^{\varepsilon}_T\right)-e^{\lambda(t-T)}{g}\left(\Bar{X}_T\right)\right| \leq C(\mathcal{P})\varepsilon.
\end{align}


Combing \eqref{N} and \eqref{CovG}, we get
\begin{equation}
\label{side1}
|u^{\varepsilon}-u|\leq C(\mathcal{P})\left(\varepsilon^{p\left(1-{1}/{\alpha_2}\right)}\vee\varepsilon \right).
\end{equation}


\section{Concluding remarks.}\label{sec-5}

In this paper, we establish an averaging principle for a class of two-time-scale stochastic control systems driven by $\alpha$-stable noise. The associated singular perturbation problem for nonlocal Hamilton-Jacobi-Bellman (HJB) equations is also investigated. By averaging over the ergodic measure of the fast component, we construct an effective stochastic control problem along with the corresponding effective HJB equation for the original multiscale system. We employ two distinct methods--a PDE approach and a probabilistic approach--to prove the convergence of the value function. Although the probabilistic method is more natural, it requires stronger regularity assumptions on the drift term of the original stochastic differential equation (SDE), in contrast to the PDE method.

Unlike the classical Gaussian case, the sample paths of jump diffusions are discontinuous; they are right-continuous with left limits. Moreover, the associated HJB equations are nonlocal. To address these challenges, we extend methodologies originally developed for SDEs with Gaussian noise and second-order HJB equations to this jump-diffusion setting.

This work has several limitations. First, the condition $1 < \alpha_i < 2$ for $i = 1, 2$ plays a crucial role in deriving the effective dynamical system. The problem of obtaining an effective reduced-dimensional system and quantifying the influence of the fast components on the slow ones remains open for the case $\alpha_i \in (0, 1)$. Second, the current analysis is restricted to additive stable LšŠvy noises. Extending these results to systems driven by multiplicative $\alpha$-stable noises poses a significant challenge. Finally, establishing the rate of convergence under more general regularity conditions on the coefficients of the slow component remains an interesting open question.

\section*{Appendix A. Proof of Lemma \ref{LiouP}.}

\begin{proof}
Without loss of generality, we assume that $V \geq 0$. For every $\eta >0$, we introduce the function $V_{\eta}(y) = V(y)-\eta w(y)$, where $w(y) = \sqrt{1 +|y|^2}$ is a Lyapunov function introduced in lemma \ref{lyaf}.
We first claim that for some fix $R>R_{0}$ and every $\eta >0$, $V_{\eta}$ is a viscosity subsolution to
\begin{equation}
  \left\{
   \begin{aligned}
    -\mathcal{L}_{2}u = & 0, \quad \text{in } B_{R}^{c}(0),  \\
    u(y) = & V_{\eta}(y), \quad \text{on } B_{R}(0).
   \end{aligned}
   \right.
\end{equation}
We prove this claim by contradiction. By definition \ref{VSD1}, we assume that there exists a point $\Bar{y} \in B^{c}_{R}(0) $ and a test function $\Bar{\varphi} \in C^{2}(\mathbb{R}^m)$ so that $\Bar{y}$ is a maximum point of $V_{\eta} - \Bar{\varphi}$ in $B_{\delta}(\Bar{y})\subset B_{R}^{c}(0)$, $V_{\eta}(\Bar{y})=\Bar{\varphi}(\Bar{y})$ and
$-\mathcal{L}_{2}\Bar{\varphi}(\Bar{y})> 0$.
By the regularity of $\Bar{\varphi}$, $V_{\eta}$, and $c(x,y)$, for every $x\in \mathbb{R}^n$ there exists a small constant $0 <\Bar{\delta} < \delta/2$ so that $-\mathcal{L}_{2}\Bar{\varphi}(y) >0$ for every $y \in B_{\Bar{\delta}}(\Bar{y})$.
Moreover, for every $y\in B^{c}_{R}$, $-\mathcal{L}_{2}w(y)>0$. Thus $V_{\eta} - \Bar{\varphi}=V-(\eta w + \Bar{\varphi})$ is a viscosity subsolution to
\begin{equation}
  \left\{
   \begin{aligned}
    -\mathcal{L}_{2}u = & 0, \quad \text{in } B_{\Bar{\delta}}(\Bar{y}),  \\
    u(y) = & h(y), \quad \text{on } B^{c}_{\Bar{\delta}}(\Bar{y}),
   \end{aligned}
   \right.
\end{equation}
where $h(y):=V_{\eta}(y)-\Bar{\varphi}(y)$. Note that $\Bar{y}$ is a strict maximum point of $V_{\eta} - \Bar{\varphi}$ on $B_{\Bar{\delta}}(\Bar{y})$, and $V(\Bar{y})=(\eta w + \Bar{\varphi})(\Bar{y})$. Thus for every $y\in B^{c}_{\Bar{\delta}}(\Bar{y})$, we have
\begin{equation}
    h(y):=V_{\eta}(y)-\Bar{\varphi}(y) < V_{\eta}(\Bar{y})- \Bar{\varphi}(\Bar{y})=0.
\end{equation}
Then the maximum principle (Corollary \ref{Maxp}) for $V_{\eta}- \Bar{\varphi}$  implies that $V_{\eta}(\Bar{y}) < \Bar{\varphi}(\Bar{y}).$
However, it is a contradiction. Thus $V_{\eta}$ is a viscosity subsolution to (\ref{NonlocalL2}) in $B_{R}^{c}$.

Since $V_{\eta}(y) \to -\infty$ as $|y|\to \infty$ for every $\eta >0$, there exist a constant $M_{\eta} > R$ so that $V_{\eta}(y) \leq \sup_{|z|\leq R}V_{\eta}(z)$ for every $|y| \geq M_{\eta}$. Using the maximum principle for $V_{\eta}$ in domain $\{R \leq |y| \leq M_{\eta}\}\subset B_{R}^{c}$, we obtain that
\begin{equation}
    V_{\eta}(y) \leq \sup_{|z|\leq R} V_{\eta}(z),\quad \forall \eta >0,\quad \forall |y|\geq R.
\end{equation}
Letting $\eta \to 0$, it follows that
\begin{equation}
    V(y) \leq \sup_{|z|\leq R} V(z),\quad \forall \eta >0,\quad \forall |y|\geq R.
\end{equation}
Thus $V(y)$ attains its global maximum at some interior point of $B_{R}(0)$. Using the maximum principle to (\ref{NonlocalL2}), $V(y)$ is a constant.

\end{proof}

\section*{Appendix B. A heuristic derivation based on the dynamic programming principle.}

\begin{proof}
In this section, we present a heuristic derivation based on the dynamic programming principle. By the definition of $u^{\varepsilon}$, we have
\begin{equation}
\begin{aligned}
    u^{\varepsilon}(t,X_{t}^{\varepsilon, t,x},Y_{t}^{\varepsilon, t,y})&:=
    u^{\varepsilon}(t,x,y)
    = \sup\limits_{v\in\mathcal{U}} \; J^{\varepsilon}(v,x,y,t)\\
    &:= \sup\limits_{v\in\mathcal{U}} \;\mathbb{E}_{(x,y)}\left[
    e^{\lambda(t-T)}g(X_{T}^{\varepsilon, t,x},Y_{T}^{\varepsilon, t,y})-
    \int_{t}^{T}L(X^{\varepsilon, t, x}_{s},Y_{s}^{\varepsilon, t,y},v_{s})e^{-\lambda(s-t)}\dd s
    \right].
\end{aligned}
\end{equation}
Now, consider the expectation term:
\begin{equation}
\begin{aligned}
    \quad \mathbb{E}_{(x,y)} & \left[
    e^{-\lambda(T-t+h-h)}g(X_{T}^{\varepsilon, t,x},Y_{T}^{\varepsilon, t,y})- \int_{t}^{t+h}L(X_{s}^{\varepsilon, t,x},Y_{s}^{\varepsilon, t,y},v_{s})e^{-\lambda(s-t)}\dd s  \right.\\
    & \quad \quad \quad \quad \quad \quad -\left.
    \int_{t+h}^{T}L(X_{s}^{\varepsilon, t,x},Y_{s}^{\varepsilon, t,y},v_{s})e^{-\lambda(s-t)}\dd s \;
    \right]\\\label{integral h}
    & = \mathbb{E}_{(x,y)}\left\{ e^{-\lambda[T-(t+h)]}\, e^{-\lambda\,h}g(X_{T}^{\varepsilon, t,x},Y_{T}^{\varepsilon, t,y})- h\, L(X_{t}^{\varepsilon, t,x},Y_{t}^{\varepsilon, t,y},v_{t})e^{-\lambda(t-t)} + o(h)   \right.\\
    & \quad \quad \quad \quad \quad \quad -e^{-\lambda\,h}\left.
    \int_{t+h}^{T}L(X_{s}^{\varepsilon, t,x},Y_{s}^{\varepsilon,t,y},v_{s})e^{-\lambda(s-(t+h)})\dd s \;
    \right\}\\
    & = \mathbb{E}_{(x,y)}\left\{ -h\,L(X_{t}^{\varepsilon,t,x},Y_{t}^{\varepsilon,t,y},v_{t}) + o(h)  +\right.\\
    & \quad \quad \quad \quad \left.+ e^{-\lambda\,h}\left[ e^{-\lambda(T-(t+h))}\,g(X_{T}^{\varepsilon, t,x},Y_{T}^{\varepsilon, t,y})- \int_{t+h}^{T}L(X_{s}^{\varepsilon, t,x},Y_{s}^{\varepsilon, t,y},v_{s})e^{-\lambda(s-(t+h)})\dd s\right]\right\}\\
    & = \mathbb{E}_{(x,y)}\bigg\{-h\,L(x,y,v_t) + o(h) + \\
    & \left. + e^{-\lambda\,h}\mathbb{E}_{(X_{t+h}^{\varepsilon, t,x}, Y_{t+h}^{\varepsilon, t,y})}\left[g(X_{T}^{t+h,X_{t+h}^{\varepsilon, t,x}},Y_{T}^{t+h,Y_{t+h}^{\varepsilon, t,y}})- \int_{t+h}^{T}L(X_{s}^{\varepsilon,t+h,X_{t+h}^{\varepsilon, t,x}},Y_{s}^{\varepsilon, t+h,Y_{t+h}^{\varepsilon, t,y}},v_{s})e^{-\lambda(s-(t+h))}\dd s\right] \right\},
\end{aligned}
\end{equation}
where the last step uses the tower property of conditional expectation:
\begin{equation*}
\mathbb{E}\left[\cdot|\sigma(X_{t}^{\varepsilon, t,x},Y_{t}^{\varepsilon, t,y})\right]=\mathbb{E}\left[\mathbb{E}\left[\cdot|\sigma(X_{t+h}^{\varepsilon, t,x},Y_{t+h}^{\varepsilon, t,y})\right]|\sigma(X_{t}^{\varepsilon, t,x},Y_{t}^{\varepsilon, t,y})\right].
\end{equation*}
Thus we obtain
\begin{equation}
 u^{\varepsilon}(t,X_{t}^{\varepsilon, t,x},Y_{t}^{\varepsilon, t,y})=\sup_{v \in U}\mathbb{E}_{(x,y)}\left[-h\,L(x,y,v) + o(h)+e^{-\lambda h} u^{\varepsilon}(t+h,X_{t+h}^{\varepsilon, t,x},Y_{t+h}^{\varepsilon, t,y})\right].
\end{equation}
This implies
\begin{equation}
\begin{aligned}
&\sup\limits_{v\in U} \; \mathbb{E}_{(x,y)}\bigg[ -h\,L(x,y,v) + o(h) + e^{-\lambda\,h}\, u^{\varepsilon}(t+h,X_{t+h}^{\varepsilon, t,x}, Y_{t+h}^{\varepsilon, t,y}) - e^{-\lambda \cdot 0}u^{\varepsilon}(t,X_{t}^{\varepsilon, t,x},Y_{t}^{\varepsilon, t,y})\bigg]\\
     &=  \sup\limits_{v\in U}\left\{-h\cdot L(x,y,v) + o(h) + \mathbb{E}_{(x,y)}\bigg[e^{-\lambda\,h}\, u^{\varepsilon}(t+h,X_{t+h}^{\varepsilon, t,x}, Y_{t+h}^{\varepsilon, t,y}) -e^{-\lambda \cdot 0}u^{\varepsilon}(t,X_{t}^{\varepsilon, t,x},Y_{t}^{\varepsilon, t,y})\bigg]\right\}\\
     &= \sup\limits_{v\in U}h\cdot\left\{-L(x,y,v) + o(1)+ \frac{1}{h}\mathbb{E}_{(x,y)}\bigg[e^{-\lambda\,h}\, u^{\varepsilon}(t+h,X_{t+h}^{\varepsilon, t,x}, Y_{t+h}^{\varepsilon, t,y}) - e^{-\lambda \cdot 0}u^{\varepsilon}(t,X_{t}^{\varepsilon, t,x},Y_{t}^{\varepsilon,t,y})\bigg]\right\}\\
     &=0.
\end{aligned}
\end{equation}
Note that
\begin{align}
    \lim\limits_{h\to 0}\, & \frac{1}{h}\mathbb{E}_{(x,y)}\bigg[e^{-\lambda\,h}\, u^{\varepsilon}(t+h,X_{t+h}^{\varepsilon,t,x}, Y_{t+h}^{\varepsilon, t,y}) - e^{-\lambda \cdot 0}u^{\varepsilon}(t,X_{t}^{\varepsilon,t,x},Y_{t}^{\varepsilon, t,y})\bigg]\\
    & = \mathbb{E}_{(x,y)}\bigg[-\lambda \, u^{\varepsilon}(t,X_{t}^{\varepsilon, t,x},Y_{t}^{\varepsilon, t,y}) + \partial_{t}u^{\varepsilon}(t,X_{t}^{\varepsilon,t,x},Y_{t}^{\varepsilon, t,y}) + \mathcal{L}^{\varepsilon}u^{\varepsilon}(t,X_{t}^{\varepsilon,t,x},Y_{t}^{\varepsilon,t,y})\bigg]\\
    & = -\lambda \, u^{\varepsilon}(t,x,y) + \partial_{t}u^{\varepsilon}(t,x,y) + \mathcal{L}^{\varepsilon}u^{\varepsilon}(t,x,y)
\end{align}
where $\mathcal{L}^{\varepsilon}$ is the generator of the process $({X_{s}^{\varepsilon}},{Y_{s}^{\varepsilon}})$.

Hence, we formally derive the nonlocal HJB equation:
\begin{align}\label{almost hjb}
    0 = \sup\limits_{v\in U}\left[-L(x,y,v)- \lambda \, u^{\varepsilon}(t,x,y) + \partial_{t}u^{\varepsilon}(t,x,y) + \mathcal{L}^{\varepsilon}u^{\varepsilon}(t,x,y) \right]
\end{align}
\end{proof}

\section*{Acknowledgments}
\medskip
\textbf{Acknowledgements} The work of Q. Zhang is supported by the China Postdoctoral Science Foundation (Grant No. 2023M740331).
The research of Y. Zhang is supported by the Natural Science Foundation of Henan Province of China (Grant No. 232300420110).

\end{document}